\documentclass[12pt]{amsart}
\usepackage{amssymb,latexsym, amssymb, amsthm,amsmath}
\usepackage{enumerate}
\usepackage{amsfonts}
\usepackage{color}
\usepackage{comment}

\usepackage{hyperref}

\newtheorem{thm}{Theorem}[section]
\newtheorem{lem}[thm]{Lemma}
\newtheorem*{proposition}{Proposition}

\theoremstyle{definition}
\newtheorem{defin}[thm]{Definition}
\newtheorem{rem}[thm]{Remark}

\makeatletter
\@namedef{subjclassname@2010}{%
  \textup{2010} Mathematics Subject Classification}
\makeatother

\numberwithin{equation}{section}

\newcommand{\N}{\mathbb{N}}
\newcommand{\R}{\mathbb{R}}
\newcommand{\C}{\mathbb{C}}
\newcommand{\Z}{\mathbb{Z}}
\newcommand{\Q}{\mathbb{Q}}

\newcommand{\T}{\mathbb{T}}

\newcommand{\vektor}[1]{\mathbf{#1}}
\newcommand{\av}{\vektor{a}}
\newcommand{\bv}{\vektor{b}}
\newcommand{\dv}{\vektor{d}}

\newcommand{\fv}{\vektor{f}}
\newcommand{\nv}{\vektor{n}}

\newcommand{\mv}{\vektor{m}}
\newcommand{\xv}{\vektor{x}}
\newcommand{\yv}{\vektor{y}}

\newcommand{\hv}{\vektor{h}}

\newcommand{\zv}{\vektor{z}}
\newcommand{\vv}{\vektor{v}}
\newcommand{\wv}{\vektor{w}}

\newcommand{\bev}{\boldsymbol{\beta}}
\newcommand{\gamv}{\boldsymbol{\gamma}}
\newcommand{\muv}{\boldsymbol{\mu}}
\newcommand{\nuv}{\boldsymbol{\nu}}
\newcommand{\nullv}{\boldsymbol{0}}
\newcommand{\einsv}{\boldsymbol{1}}

\newcommand{\Ma}{{\mathfrak M}}
\newcommand{\Ua}{{\mathfrak U}}

\begin{document}

\title[Translation invariant quadratic forms]{Translation invariant quadratic forms \\ in dense sets}

\author[Eugen Keil]{Eugen Keil}

\address{Mathematical Institute \\ University of Oxford\\
Mathematical Institute \\ 24-29 St Giles'\\ OX1 3LB Oxford \\ United Kingdom}

\email{Eugen.Keil@maths.ox.ac.uk}

\date{\today}


\maketitle

\begin{abstract}
We generalize Roth's theorem on three term arithmetic progressions to translation invariant quadratic forms in
at least $17$ variables. We use Fourier-analysis, restriction theory, uniformity norms and 
Roth's density increment method to show quantitative estimates for subsets of the integers 
without any non-trivial solutions.
\end{abstract}

\section{Introduction}

In 1953 Roth \cite{Roth} proved his theorem on $3$-term arithmetic progressions in dense sets.
It states that a subset $\mathcal{A} \subset \{1,2,\ldots,N\}$ with no arithmetic
progresions of the form $x,x+h,x+2h$ with $h \geq 1$ cannot be too large. His theorem 
gives the bound $|\mathcal{A}| \leq CN (\log \log N)^{-1}$ for some constant $C \geq 1$.
In other words, it is not possible to avoid $3$-term arithmetic progressions as long as the density 
of the set $\mathcal{A}$ is big enough.

Arithmetic progressions can also be described as solutions to translation invariant 
equations (see explanation at the end of the introduction). 
In the case of $3$-term progressions we have the equation $x_1-2x_2+x_3 = 0$.
Roth \cite{Roth2} went on to prove a version of his theorem for solutions to translation invariant 
linear systems in $k$ equations with at least $2k+1$ variables. 
By recent work of Gowers \cite{Gow} we can now solve 
translation invariant systems with as few as $k + 2$ variables in sets $\mathcal{A}$ of cardinality 
at least $C_kN (\log\log N)^{-c_k}$ for some $C_k,c_k > 0$.

The aim of this work is to combine the ideas of Gowers \cite{Gow}, Green \cite{Gr},
Liu \cite{Liu}, Roth \cite{Roth} and the previous work of the author \cite{Keil} to give 
a version of Roth's theorem for quadratic forms.

\begin{thm} \label{Thm1}
Let $\xv^TQ\xv = 0$ be a translation invariant quadratic equation in $s \geq 17$ variables. 
Assume that it has a non-singular real solution, but only trivial solutions when the variables
are restricted to $\mathcal{A} \subset \{1,2,\ldots,N\}$. 
Then $|\mathcal{A}| \leq C_Q N (\log\log N)^{-c}$
for some $c,C_Q > 0$.
\end{thm}

Theorem \ref{Thm1} will follow from the more precise Theorem \ref{Main-thm}. 
In most cases we only need $s \geq 10$ variables as in the work of Liu \cite{Liu}. 
The bound $s \geq 17$ is a worst case scenario and can certainly be improved by a more
complicated analysis.

The next observation is that the exponent $c$ in $C_Q N(\log\log N)^{-c}$
is independent of the quadratic form $Q$ and the number of variables $s$ (we use 
the letter $Q$ interchangeably for the quadratic form and the underlying matrix). 
If one would allow such a dependence, it is possible to derive the above theorem 
(even for all $s \geq 5$) by the methods of Gowers \cite{Gow} as follows. 
Take any integer $\yv$ with $Q(\yv) = 0$ (see Lemma \ref{Standard-quad-result}) and consider the 
patterns $(a+qy_1,\ldots,a+qy_s)$ for $a \in \Z$ and $q \in \N$.
Then the approach in \cite{Gow} shows that a set that doesn't contain
any of those patterns, must have a density bounded by $C_{Q} N(\log\log N)^{-c_s}$.

We want to point out that it is very likely that the methods of this work 
can be adapted to give an asymptotic count of the number of solutions to $Q(\xv) = 0$ for 
$x_i \in \mathcal{B}$ for some relatively structured set 
$\mathcal{B}$ such as the primes, for example. This is not possible by relying on
the work of \cite{Gow} or other purely additive combinatorial results from the linear theory.
This explains some of the motivation behind Theorem \ref{Thm1}.


Recent years have seen huge advances in our understanding of linear equations in primes.
Work of Green \cite{Gr} and Green and Tao \cite{GrTao2} introduced the concept of a 
`pseudorandom measure', which led to amazing new developments in the linear theory \cite{GrTao3}.
These results can be used to find prime solutions for general diophantine equations, 
such as in recent work of Br\"udern, Dietmann, Liu and Wooley \cite{BDLW} on the 
Birch-Goldbach problem.

If one is interested in asymptotics, on the other hand, one has to deal with the
non-linear theory directly.
Recent work on prime solutions for quadratic forms by Liu \cite{Liu} uses
a variant of the circle method to deal with a large class of quadratic forms 
in at least ten variables and provides one of the main ideas for this work.

Previous work on diagonal translation invariant forms was carried out 
by Smith \cite{Smith}, who considers the system
\begin{align*}
\lambda_1 x_1^2 + \lambda_2x_2^2 + \ldots + \lambda_s x_s^2 & = 0, \\
\lambda_1 x_1 + \lambda_2x_2 + \ldots + \lambda_s x_s & = 0
\end{align*}
with $\lambda_1 + \ldots + \lambda_s = 0$ in $s \geq 9$ variables. 
The author simplified Smith's approach in \cite{Keil} and reduced 
the number of variables down to $s \geq 7$. The methods of \cite{Keil} play a significant
role in the development of this work and we cite several results from \cite{Keil} 
to simplify the exposition here. Readers interested in the restriction theory part
of the argument are adviced to have a look at \cite{Keil} for more explanations. 

Before proceeding to the main part of the paper, we want to give the reader
the opportunity to gain some intuition about the consequences of assuming 
`translation invariance' in the context of quadratic forms. 
For linear equations, this geometric condition translates
into the arithmetic statement that the sum of the coefficients in each equation is zero.
This is also true for the diagonal quadratic system considered above.
For a quadratic form $Q(\xv) := \xv^T Q \xv = 0$ with a symmetric matrix $Q \in \Z^{s \times s}$ 
\emph{translation invariance} means that $Q(\xv + h\einsv) = Q(\xv)$ for all
$\xv \in \Z^s$ and $h \in \Z$, where $\einsv = (1,\ldots,1)^T \in \Z^s$. 

This implies that $Q(h\einsv) = Q(\nullv) = 0$. 
We call the multiples of $\einsv$ the \emph{trivial solutions} to our quadric.
If we expand $Q(\xv + h\einsv)$, we obtain
\begin{align*}
Q(\xv) = Q(\xv + h\einsv) = Q(\xv) + 2h\xv^T Q \einsv + h^2Q(\einsv).
\end{align*}
It follows that $Q\einsv = \nullv$ and it is easy to check that it is a sufficient condition as well.

Another way of looking at this issue is to set $h = -x_s$. Then we get
$Q(\xv - x_s \einsv) = Q(\xv)$ and, therefore, any translation invariant quadratic form can be written in the form
$Q'(x_1-x_s,\ldots,x_{s-1}-x_s)$ for some arbitrary quadratic form $Q'$ in $s-1$ variables.

To prove Theorem \ref{Thm1}, it is not always necessary to assume translation invariance,
as can be seen from considering only the first equation from the diagonal
quadratic system above, but the condition $Q(\einsv) = 0$ is clearly necessary.
Otherwise, we can choose $\mathcal{A}$ as the set of numbers congruent to
one modulo $n$, where $n$ is some large number (dependent only on the coefficients of $Q$)
and obtain a contradiction.

{\bf Acknowledgements:}\\
This work is part of the author's Ph.D. thesis and he would like to thank his 
supervisor Trevor Wooley for proposing the topic and constant encouragement.
The doctoral studies of the author were partially supported by the EPSRC.

\section{Notation and General Discussion} \label{Not-Out}

First we remind the reader about some standard notation.
We write $e(x) = \exp(2\pi i x)$ and use $f = O(g)$ to express that $|f| \leq Cg$ for 
some constant $C > 0$ and similarly Vinogradov's notation $f \ll g$. We indicate dependencies 
on parameters by subscripts as in $O_p(N^s)$ or $\ll_{P, \epsilon}$, for example.
The parameter $N \in \N$, governing the size of the variables $x_i$ should be thought of as large
and we write $[1,N]$ as abbreviation for the interval $\{1,2,\ldots,N\}$.
The set $\mathcal{A}$ is always a subset of $[1,N]$ with density 
$\delta = |\mathcal{A}|/N$ and indicator function $1_{\mathcal{A}}$.
The balanced function $f(x) = 1_{\mathcal{A}}(x) - \delta$ plays an important role 
at various places in this paper.

Bold face letters such as $\xv$ denote vectors with components $x_i$ and
inequalities such as $\xv \leq N$ or $\xv \leq \yv$ should be understood componentwise.
A sum over natural numbers starts at one, if not otherwise indicated.
The symbol $\T$ is used to refer to the `circle' $\R/\Z$ with the circle norm
$\|\alpha\| := \min\{|\alpha - z|: z \in \Z\}$,
the distance of $\alpha \in \R$ to the nearest integer.

We don't want to distinguish between quadratic forms that are related
by a simple renaming of variables. Given two matrices $A$ and $B$ in $\R^{s \times s}$ 
we say they are \emph{permutation-equivalent} if 
\begin{align*}
A = P^TBP
\end{align*}
for an invertable matrix $P \in \{0,1\}^{s \times s}$.

To explain the general structure of the work and the main theorem, we consider
the following property of quadratic forms.

\begin{defin}[Off-diagonal rank]
For a symmetric matrix $Q \in \R^{s \times s}$ we consider matrices $M$
that are permutation equivalent to $Q$ and write them in the form
\begin{align*}
M = \begin{pmatrix} 
  A  & B\\ 
  B^T & C  
\end{pmatrix}
\end{align*}
for some matrices $A,B$ and $C$. Then the \emph{off-rank} $r$ of $Q$ is defined as 
\begin{align*}
r = \max\{\mbox{rank}(B): M \mbox{ is permutation equivalent to } Q \}.
\end{align*}
In other words, this is the maximal rank of a submatrix in Q, that doesn't
contain any diagonal elements.
\end{defin}

The off-rank $r$ of a matrix determines the treatment of the corresponding quadratic equation. 
While for $r \geq 5$ we can apply the bilinear sum method 
inspired by the work of Liu \cite{Liu} on prime solutions for quadratic forms, 
we need a more complicated approach for $r \leq 4$, based on 'partial diagonalisation'.
This leads to the following main result of this work.

\begin{thm} \label{Main-thm}
Let $Q \in \Z^{s \times s}$ be symmetric with $Q \cdot \einsv = \nullv$ and off-rank~$r$.
Assume that $Q(\xv) = 0$ has a non-singular real solution and assume that 
$s \geq 5 + 3r$ for $1 \leq r \leq 4$ and $s \geq 10$ for $r \geq 5$. 
If there are only trivial solutions, when the variables
are restricted to $\mathcal{A} \subset [1,N]$, then $|\mathcal{A}| \ll_Q N (\log\log N)^{-c}$
for some $c > 0$ independent of $Q$.
\end{thm}

\emph{Remark.}
Almost all quadratic forms in at least 10 variables have off-rank at least five. 
The exceptional cases lie in a lower-dimensional submanifold, meaning that we need
only $s \geq 10$ for almost all quadratic forms.\\

Apart from the off-rank problem, there is the general issue of positive
definiteness that needs to be addressed in the case of quadratic forms.
We saw in the introduction that any translation invariant quadratic form can be written in
the form $Q(\xv) = Q(\xv - x_s\einsv) = Q'(x_1-x_s,\ldots,x_{s-1}-x_s)$, where the matrix 
for the quadratic form $Q'$ is given by the upper left submatrix of size $(s-1)\times(s-1)$ in $Q$.
Now we can diagonalize $Q'(z_1,\ldots,z_{s-1}) = 0$ over $\Z$, where $z_i = x_i - x_s$.
This can be done by completing the square successively and then multiplying by a 
suitable integer to ensure that the rational coefficients that appear during this
process become integers again. We get for some $\lambda_j \in \Z$ an equation of the form
\begin{align} \label{diagonal-equation}
\lambda_1 y_1^2 + \ldots + \lambda_{s-1} y_{s-1}^2 = 0,
\end{align}
where the $y_i$ are independent linear forms in the variables $z_i$ or equivalently,
translation invariant linear forms in the variables $x_1,\ldots,x_s$.
If we have $\lambda_i \geq 0$ for all $1 \leq i \leq s$, then all real solutions to this
quadric are singular and this case is excluded by the assumptions in Theorem \ref{Main-thm}.
This is also true for the case when $\lambda_i \leq 0$ for all $1 \leq i \leq s$.

In the remaining cases we have at least one negative and at least one
positive coefficient $\lambda_i$. The existence of coefficients of both signs 
is equivalent to the existence of a non-singular real solution.
If the number $t$ of non-zero coefficients $\lambda_i$ is less than five, 
we can get problems with $p$-adic solubility as in the example
\begin{align*}
y_1^2 + y_2^2 - 3(y_3^2 + y_4^2) = 0,
\end{align*} 
where we have only the zero solutions modulo eight. In this case we 
consider the linear system $y_k = 0$ for $1 \leq k \leq t \leq 4$ instead.
This case is covered in Section \ref{Linsystem}.

For most quadratic forms we have $t \geq 5$ and have the following Lemma.

\begin{lem} \label{Standard-quad-result}
If $s \geq 5$, $d_i \neq 0$ for $1 \leq i \leq s$ and not all coefficients 
$d_i$ in
\begin{align} \label{diag-quad}
d_1x_1^2 + \ldots + d_sx_s^2 = 0,
\end{align}
have the same sign, then \eqref{diag-quad} has $C(\dv) N^{s-2} + o(N^{s-2})$ solutions 
with $x_i \in [1,N]$ for some constant $C(\dv) > 0$ dependent on the coefficients $d_i$.
\end{lem}

\begin{proof}
The proof is essentially given in Chapter 8 of \cite{Dav}. 
Chapter 2 of \cite{Vau} also contains
all the necessary estimates to deduce the result. 
Another approach can be found in the recent paper \cite{HB2}.
\end{proof}

The last ingredient for the proof of Theorem \ref{Main-thm} are
$L^p$-properties for certain exponential sums.
For a function $g: \N \to \C$ and $\mathcal{A} \subset [1,N]$ we define
\begin{align} \label{eq-Vg}
L_{g}(\alpha) = \sum_{z \leq N} g(z) e(\alpha z) \mbox{\quad and \quad}
V_g(\alpha, \beta) = \sum_{x \leq N} g(x) e(\alpha x^2 + \beta x)
\end{align}
and write $L_{\mathcal{A}}(\alpha)$ or 
$V_{\mathcal{A}}(\alpha, \beta)$ in the case $g = 1_{\mathcal{A}}$ and 
$L(\alpha), V(\alpha, \beta)$ for the sums without any weight $g$.
We use Appendix \ref{AppC} to derive two useful $L^p$ estimates 
along the lines of \cite{Keil}.

The general structure of the paper is as follows.
In Section \ref{Bilinear} we treat the case $r \geq 5$ with a refinement of Liu's method \cite{Liu}.
Appendix \ref{AppA} provides the new necessary ingredient, a sharp `Vinogradov lemma'.
In Sections \ref{PD} to \ref{Konvex} we consider the non-degenerate part of the case $r \leq 4$
and use `convexity' methods to simplify our mean-value integrals to deduce 
a correlation estimate for the exponential sums in \eqref{eq-Vg} with $g = f$.
In Section \ref{Density}, we use the correlation estimates to prove Theorem \ref{Main-thm}.
Section \ref{Linsystem} finally deals with the degenerate cases, where we can 
extract a linear subsystem from $Q$, which can be treated by Gowers' theory \cite{Gow}.
Appendix \ref{AppB} provides a short proof for the uniformity norm estimate for completeness.

\section{The Bilinear Sum Method} \label{Bilinear}

The main goal of this section is to deduce correlation estimate \eqref{CorS-est} in the case $r \geq 5$.
We follow Liu \cite{Liu} and simplify his approach by removing the `geometry of numbers' argument.
For a quadratic form $Q(\xv) = \xv^TQ\xv$ define the exponential sum
\begin{align*}
S_g(\alpha) = \sum_{\xv \leq N} g(\xv) e(\alpha Q(\xv)),
\end{align*}
where $g(\xv) = \prod_{i = 1}^s g_i(x_i)$ and $|g_i| \leq 1$.
The main technical result in this section is an $L^p$-estimate for this quadratic exponential sum.

\begin{thm} \label{Lp-thm1}
Let $Q$ have off-rank $r \geq 1$. Then for $p > 4/r$ we have
\begin{align*}
\int_0^1 |S_g(\alpha)|^p \, d\alpha \ll_p N^{ps-2}.
\end{align*}
Assume the $L^1$-bound $\sum_{x \leq N} |g_i(x)| \leq 2\delta N$ and $r \geq 5$. 
Then $p > 4/5$ implies
\begin{align*}
\int_0^1 |S_g(\alpha)|^p \, d\alpha \ll_{p,s} \delta^{(s-10)p}N^{ps-2}.
\end{align*}
\end{thm}

\begin{proof}
The matrix $Q$ is permutation equivalent to a matrix of the form 
\begin{align*}
\begin{pmatrix} 
  A   & R\\ 
  R^T & B 
\end{pmatrix}
\end{align*}
with rank$(R)=r$. To simplify notation, we can also assume 
that the first $r$ rows of $B$ are linearly independent.
Decompose the variables $\xv = (\xv_a,\xv_b)$ accordingly into two blocks of sizes $a$ and $b$
with at least $r$ variables each.
Then the quadratic form $Q(\xv)$ has the representation
\begin{align*}
Q(\xv) = \xv_a^T A \xv_a + 2 \xv_a^T R \xv_b + \xv_b^T B \xv_b,
\end{align*}
Write $g(\xv_a) = \prod_{i = 1}^a g_i(x_i)$ and $g(\xv_b) = \prod_{i = a+1}^s g_i(x_i)$.
Then we have the estimate
\begin{align*}
|S_g(\alpha)| = & \Big|\sum_{\xv_a \leq N} \sum_{\xv_b \leq N} g(\xv_a) g(\xv_b) 
e(\alpha (\xv_a^T A \xv_a + 2 \xv_a^T R \xv_b + \xv_b B \xv_b)) \Big|\\
\leq & \sum_{\xv_a \leq N} \Big|\sum_{\xv_b \leq N} g(\xv_b) e(\alpha (2 \xv_a^T R \xv_b + \xv_b B \xv_b)) \Big|\\
\leq & N^{a/2} \Big(\sum_{\xv_a \leq N} \Big|\sum_{\xv_b \leq N} g(\xv_b)
e(\alpha (2 \xv_a^T R \xv_b + \xv_b B \xv_b)) \Big|^2 \Big)^{1/2}
\end{align*}
by the inequality of Cauchy-Schwarz. 
The expression in the parentheses on the right hand side is
\begin{align*}
& \sum_{\xv_b \leq N} \sum_{\xv'_b \leq N} g(\xv_b) \overline{g(\xv'_b)}\sum_{\xv_a \leq N} 
e(\alpha (2 \xv_a^T R (\xv_b-\xv'_b) + \xv_b B \xv_b - \xv'_b B \xv'_b))\\
\leq & \sum_{\xv_b \leq N} \sum_{\xv'_b \leq N} \Big|\sum_{\xv_a \leq N} 
e(2 \alpha \xv_a^T R (\xv_b - \xv'_b)) \Big| \\
\leq & \sum_{\xv_b \leq N} \sum_{\xv'_b \leq N} \prod_{i =1}^a 
\min\left(N, \|2 \alpha R_i (\xv_b - \xv'_b)\|^{-1} \right)
\end{align*}
where $R_i$ is the $i$th row of $R$.

Consider the equations $y_i = 2 R_i (\xv_b - \xv'_b)$. The variables $y_i$ can vary over
an interval $[-PN,PN]$ for some constant $P$ depending on the size of the coefficients of $R$. 
Since $R$ has rank $r$, the system of equations 
$y_i = 2 R_i (\xv_b - \xv'_b)$ has $O(N^{2b-r})$ solutions for given $|y_1|,\ldots,|y_r| \leq PN$.
We bound the other $a-r$ factors ($i > r$) trivially by $N$ and obtain the bound
\begin{align*}
|S_g(\alpha)|^2 \ll N^{2b+2a-2r} \sum_{|\yv| \leq M} \prod_{i =1}^r \min\left(N, \|\alpha y_i \|^{-1} \right).
\end{align*}
Since $a + b = s$ and the inner expression factors into $r$ independent sums, we can apply
Lemma \ref{Weyl-est}. Write $\alpha = a/q + \beta$ with $|\beta| \leq 1/(qN)$ for some $q \leq N$ 
by Dirichlet's approximation theorem and define
\begin{align} \label{eq-K}
K(\alpha) = \Big(N \log q + \min\Big\{ \frac{N^2}{q}, \frac{|\log(|\beta| N^2)| + 1}{|\beta|q} \Big\} \Big)^{1/2}.
\end{align}
If the representation of $\alpha$ turns out to be non-unique, 
we take the minimal value attained by the various functions on the right hand side of \eqref{eq-K}.
With this definition, we get from Lemma \ref{Weyl-est} the estimate
\begin{align} \label{Sf-est}
|S_g(\alpha)| \ll N^{s-r} K(\alpha)^r.
\end{align}
To obtain the first $L^p$-bound in Theorem \ref{Lp-thm1} we 
prove the following useful lemma.

\begin{lem} \label{K-bound}
For $p > 4$ we have $\int_0^1 |K(\alpha)|^p\,d\alpha \ll_p N^{p-2}$. 
\end{lem} 

\begin{proof}
We decompose $[0,1]$ according to the Dirichlet approximations, which give us 
\begin{align*}
\int_0^1 |K(\alpha)|^p \, d\alpha 
\ll \sum_{q \leq N} \sum_{a = 1}^q \int_{|\beta| \leq 1/(qN)} |K(a/q + \beta)|^p \, d\beta,
\end{align*}
where the summation in $a$ is only over the elements with $(a;q) = 1$.
We decompose the integration over $\beta$ further into
sets with $|\beta| \leq N^{-2}$ and a dyadic decomposition $|\beta| \in (2^iN^{-2},2^{i+1}N^{-2}]$
for $0 \leq i \leq \log_2(Nq^{-1})$.
For $|\beta| \leq N^{-2}$ we get by \eqref{eq-K} the contribution
\begin{align*}
\sum_{q \leq N} \sum_{a = 1}^q N^{-2} \Big(N \log q + N^2 q^{-1} \Big)^{p/2} \ll N^{p-2}.
\end{align*}
On each dyadic part $|\beta| \in (2^iN^{-2},2^{i+1}N^{-2}]$, we obtain the contribution
\begin{align*} 
2^i N^{-2} \Big(N \log q + N^2 (i + 1)2^{-i} q^{-1} \Big)^{p/2}.
\end{align*}
Apply the bound $|x+y|^{p/2} \ll |x|^{p/2} + |y|^{p/2}$ and sum over $i$ to arrive at
\begin{align*}
\sum_{q \leq N} \sum_{a = 1}^q \Big( (Nq)^{-1} (N \log q)^{p/2} + N^{-2} (N^2/q)^{p/2} \Big) \ll N^{p-2}.
\end{align*}
\end{proof}
Now the first part of Theorem \ref{Lp-thm1} follows from \eqref{Sf-est}.
For the second part we have $r \geq 5$. By permutation equivalence we may assume, that
\begin{align} \label{eq-matrix2}
Q = \begin{pmatrix}
A & R & X\\
R^T & B & Y\\
X^T & Y^T & C
\end{pmatrix},
\end{align}
where $R \in \R^{5 \times 5}$ is a full rank matrix.
We can divide the variable vector $\xv$ into $(\xv_a,\xv_b,\xv_c)$, where $\xv_a$ and $\xv_b$ contain
five variables and $\xv_c$ the remaining $s-10$ variables.
Then we have the bound 
\begin{align}  \label{Sg-extra-est}
|S_g(\alpha)| \leq \sum_{\xv_c \leq N} |g_c(\xv_c)| \Big|\sum_{\xv_a \leq N}
\sum_{\xv_b \leq N} g_a(\xv_a)g_b(\xv_b) e(\alpha Q(\xv)) \Big|,
\end{align}
where the functions $g_a,g_b$ and $g_c$ are defined in a similarlar way as above.
If we expand the quadratic form $Q(\xv)$ by use of \eqref{eq-matrix2} 
and the decomposition $(\xv_a,\xv_b,\xv_c)$ we get
\begin{align} \label{eq-expand-matrix}
\xv^TQ\xv = (\xv_a^T,\xv_b^T) \begin{pmatrix}A & R\\R^T & B\end{pmatrix} 
\begin{pmatrix}\xv_a \\\xv_b\end{pmatrix} + 
2 \xv_a^T X \xv_c + 2 \xv_b^T Y \xv_c + \xv_c^T C \xv_c.
\end{align}
The last summand depends only on $\xv_c$ and will disappear due to the absolute
value signs in \eqref{Sg-extra-est}. The other two parts which contain $\xv_c$ 
can be seen as a sum of linear forms $L_{i,\xv_c}(x_i)$ for $i \leq 10$.
Therefore, the absolute value of the inner sum in \eqref{Sg-extra-est} 
can be seen as $|S_h(\alpha)|$ for a new function
\begin{align*}
h_{\xv_c}(\xv_a,\xv_b) = \prod_{i=1}^{10} g_i(x_i) e(\alpha L_{i,\xv_c}(x_i))
\end{align*}
and the quadratic form corresponding to the $10\times 10$ matrix in \eqref{eq-expand-matrix}.

For $p > 4/5$ we use \eqref{Sf-est} and obtain the bound
\begin{align*}
& \int_0^1 |S_g(\alpha)|^p \, d\alpha \leq \int_0^1 \Big|\sum_{\xv_c \leq N} |g_c(\xv_c)||S_h(\alpha)|\Big|^p \, d\alpha\\
\ll & \int_0^1 \Big|\sum_{\xv_c \leq N} |g_c(\xv_c)|N^{5} K(\alpha)^{5}\Big|^p \, d\alpha
\ll \Big(\sum_{\xv_c \leq N} |g_c(\xv_c)| \Big)^p N^{10p - 2}.
\end{align*}
The result now follows from the assumption $\sum_{x \leq N} |g_i(x)| \leq 2\delta N$.
\end{proof}

Having proven Theorem \ref{Lp-thm1}, we can deduce a correlation estimate for the exponential sum $S_g(\alpha)$
in the cases $r \geq 5$.
From the discussion of Section \ref{Not-Out}, the assumption in Theorem \ref{Main-thm}, we have
by Lemma \ref{Standard-quad-result} the lower bound
\begin{align*}
\int_0^1 S(\alpha) \, d\alpha \gg N^{s-2},
\end{align*}
where $S(\alpha)$ is the exponential sum with $g = 1$. On the other hand, 
if we write $S_{\mathcal{A}}(\alpha)$ for the exponential sum with the indicator function
$g = 1_{\mathcal{A}}$ of the set $\mathcal{A}$, we obtain
\begin{align*}
\int_0^1 S_{\mathcal{A}}(\alpha) \, d\alpha = \delta N
\end{align*}
since there are only trivial solutions in the set $\mathcal{A}$.
By comparing those two quantities, we derive
\begin{align*}
\int_0^1 |S_{g}(\alpha)| \, d\alpha \gg \delta^s N^{s-2}
\end{align*}
for $g(\xv) = 1_{\mathcal{A}^s}(\xv) - \delta^s$ with $\delta = |\mathcal{A}|/N$
as long as $N \gg_Q \delta^{-2}$, say.
While this function $g$ doesn't satisfy the conditions of Theorem \ref{Lp-thm1}, 
we can write it as a finite sum $g(\xv) = \sum_{i=1}^s f_i(\xv)$ of functions
\begin{align} \label{fi-def}
f_i(\xv) =  (1_{\mathcal{A}}(x_i) - \delta) \delta^{i-1} \prod_{j > i} 1_{\mathcal{A}}(x_j),
\end{align}
that factor into a product of factors $h_j$. Each of those satisfies the $L^1$-condition
$\sum_{x \leq N} |h_j| \leq 2 \delta N$. Therefore, by part two of Theorem \ref{Lp-thm1} 
we have for some $i \leq s$ the estimate
\begin{align*}
\delta^s N^{s-2} & \ll \sup_{\alpha} |S_{f_i}(\alpha)|^{1-p} \int_0^1 |S_{f_i}(\alpha)|^p \, d\alpha \\
& \ll \sup_{\alpha} |S_{f_i}(\alpha)|^{1-p}  (\delta N)^{(s-10)p} N^{10p-2}.
\end{align*}
Now set $p = 8/9 > 4/5$, for example, and deduce the correlation estimate
\begin{align} \label{CorS-est}
\sup_{\alpha} |S_{f_i}(\alpha)| \gg \delta^{s+80} N^s.
\end{align}
This correlation estimate is used in Section \ref{Density} to run the usual density increment
method of Roth \cite{Roth}.

\section{Partial diagonalisation} \label{PD}

The correlation estimate for the case $r \leq 4$ requires 
more work and the problem splits into several subcases.
In this section we take the first step and 
prove a structure result for quadratic forms of low off-rank 
to extract a partial diagonal structure.

By permutation equivalence, we may assume that
\begin{align} \label{eq-matrix-decomp}
Q = \begin{pmatrix}
A & R & M\\
R^T & B & N\\
M^T & N^T & C
\end{pmatrix},
\end{align}
where $R \in \R^{r \times r}$ is a full rank matrix.
We want to show that we can `diagonalize' $C$ by adding at most $r$ linear equations to $Q(\xv)=0$.
Since the size of the matrix $C$ in \eqref{eq-matrix-decomp} is $s - 2r$, 
we can hope that methods for diagonal quadrics can provide solutions, 
as long as $s$ is not too small compared to $r$.
We begin by stating a simple lemma.

\begin{lem} \label{lem-deter}
Let $R \in \R^{r \times r}$ be a full rank matrix and $\vv,\wv \in \R^r$, then there is exactly one
$c \in \R$ such that the matrix
\begin{align*}
\begin{pmatrix}
R & \vv\\ \wv^T & c
\end{pmatrix}
\end{align*}
has rank $r$, namely $c = \wv^T R^{-1} \vv$.
\end{lem}

\begin{proof}
Multiplication by an invertable matrix on the left leads to the relation
\begin{align*}
\begin{pmatrix}
R^{-1} & 0\\ \wv^TR^{-1} & -1
\end{pmatrix}
\begin{pmatrix}
R & \vv\\ \wv^T & c
\end{pmatrix}
= \begin{pmatrix}
E_r & R^{-1}\vv \\ 0 & \wv^T R^{-1}\vv - c
\end{pmatrix},
\end{align*}
where $E_r \in \Z^{r \times r}$ is the identity matrix. This implies the result.
\end{proof}

Now we apply Lemma \ref{lem-deter} to `diagonalize' $C$.

\begin{lem} \label{lem-C-diag}
Let $M,N,C,R$ be the matrices from \eqref{eq-matrix-decomp}.
Then 
\begin{align*}
C = N^TR^{-1}M + D,
\end{align*} 
where $D$ is a diagonal matrix.
\end{lem}

\begin{proof}
For a fixed element $c_{ij}$ from $C$ with $i \neq j$ consider the matrix
\begin{align*}
\begin{pmatrix}
R & \mv_j\\ \nv_i^T & c_{ij}
\end{pmatrix},
\end{align*}
where $\mv_j,\nv_i^T$ are the $j$th column and $i$th row from the matrices $M$ and $N^T$
respectively. This is a submatrix of $Q$ that lies completely off-diagonal and, therefore,
cannot have rank more than $r$. By Lemma \ref{lem-deter}, we get $c_{ij} = \nv_i^T R^{-1} \mv_j$,
which gives the desired claim.
\end{proof}

Lemma \ref{lem-C-diag} says that there is a diagonal matrix $D$ such that the rows 
of $C-D$ are linear combinations of rows of $M$.
The next step is to find a common basis of linear forms, that span the rows of $M$, $N$ and $C-D$. 

\begin{lem} \label{lem-H}
Let $M,N$ be as in \eqref{eq-matrix-decomp}.
There is a matrix $H \in \Z^{t \times (s-2r)}$ with $t \leq r$ and
linearly independent rows such that $M = \hat{M}H$ and $N = \hat{N}H$
for some matrices $\hat{M},\hat{N} \in \Q^{r \times t}$.
\end{lem}

\begin{proof}
The submatrix consisting of $M$ and $N$, as in \eqref{eq-matrix-decomp}, namely
\begin{align*}
\begin{pmatrix}
M\\ N
\end{pmatrix},
\end{align*}
has rank $t$ with $t \leq r$, since $Q$ has off-rank $r$. One can find
$t$ rows which span the rowspace of this matrix. Arrange these into a single matrix $H$.
Then we can write
\begin{align*}
\begin{pmatrix}
M\\ N
\end{pmatrix}
= \begin{pmatrix}
\hat{M}\\ \hat{N}
\end{pmatrix}
\cdot H,
\end{align*}
for some matrices $\hat{M},\hat{N} \in \Q^{r \times t}$.
\end{proof}

Combining Lemma \ref{lem-C-diag} and Lemma \ref{lem-H}, we obtain
\begin{align} \label{diag-eq}
C = H^T\hat{N}^TR^{-1}\hat{M}H + D.
\end{align}

This insight is sufficient to diagonalize $C$. But there is another
technical rank condition on $H$ that we need for later parts of this chapter, 
which requires another step of linear algebra. Consider the following property for matrices.
\begin{itemize}
\item[\bf AP:] \emph{If we remove any column from $H \in \Z^{t \times g}$, then the remaining matrix contains 
two disjoint $t \times t$ non-singular submatrices.}
\end{itemize}

This is the key property for the application of the classical circle method.
We can always ensure that it is satisfied by passing to a submatrix and the use
of the following lemma.

\begin{lem} \label{vec-lem}
Let $A$ be a $t \times m$ matrix over a field $K$ and $q$ be a positive integer.
Then either $A$ includes $q$ disjoint $t \times t$ non-singular submatrices or
all but at most $q(t-d)-1$ columns are contained in a $d$-dimensional subspace for
some $0 \leq d \leq t-1$.
\end{lem}

\begin{proof}
This is a special case of Proposition 6.45 in \cite{Aig} and 
a proof may also be found in \cite{LPW}.
\end{proof}

Now, either $H$ satisfies the property [AP] or we can remove the `bad' column 
and apply Lemma \ref{vec-lem} with $q = 2$. This gives us that at most $2(r-d)-1$
columns are not contained in a $d$-dimensional subspace for some $0 \leq d \leq r-1$.
Remove these other exceptional columns as well and we obtain a matrix $\overline{H}$ with rank
at most $d$ and $w \geq s-2r-2(r-d)$ columns.
Rename variables if necessary and write $H = \begin{pmatrix} H' & \overline{H} \end{pmatrix}$. 
By splitting the other matrices accordingly, we arrive at the form
\begin{align*}
Q = \begin{pmatrix}
A & R & M' & \overline{M}\\
R^T & B & N' & \overline{N}\\
M'^T & N'^T & C_{11} & C_{12}\\
\overline{M}^T & \overline{N}^T & C_{12}^T & C_{22}
\end{pmatrix},
\end{align*}
where $\overline{M} = \hat{M} \overline{H}$, $\overline{N} = \hat{N} \overline{H}$, $M' = \hat{M}H'$, $N' = \hat{N}H'$. 
The matrix $C$ splits according to formula \eqref{diag-eq} into parts $C_{11} ,C_{12}, C_{12}^T, C_{22}$ with 
$C_{12} = H'^T\hat{N}^TR^{-1}\hat{M}\overline{H}$, for example.

If we choose $d$ minimal in the above procedure, we end up with a matrix $\overline{H}$
that contains $d$ linearly independent rows, which (as a matrix) satisfy property [AP]. 
Otherwise, we could apply Lemma \ref{vec-lem} again and obtain
the same result for a smaller value of $d$.

Now we decompose our variables suitable for this decomposition.
Since $\overline{H}$ has $w$ columns, we have $C_{22} \in \Z^{w \times w}$ with $w \geq s-4r+2d > 0$
if $s \geq 5 + 3r$.
Call the first $s-w$ variables $\yv = (y_1,\ldots,y_{s-w})$ and the remaining $w$ variables $\xv = (x_1,\ldots,x_w)$.
The original equation $\xv^TQ\xv=0$ decomposes into 
\begin{align} \label{eq-mat1}
\yv^T 
\begin{pmatrix}
A & R & M' \\
R^T & B & N'\\
M'^T & N'^T & C_{11}
\end{pmatrix} \yv + 2 \yv^T 
\begin{pmatrix}
\overline{M}\\
\overline{N}\\
C_{12}
\end{pmatrix} \xv + \xv^TC_{22}\xv = 0.
\end{align}
For $1 \leq i \leq d$ we add linear equations 
\begin{align*}
\mu_{i1}x_1 + \ldots + \mu_{it}x_w = h_i,
\end{align*}
where the coefficients $\mu_{ij} \in \Z$ correspond
to $d$ linearly independent rows of $\overline{H}$ that satisfy [AP].
The variables $h_i$ may have any integer value, but due to the restrictions on the $x_j$
they will range over a bounded interval of size $O_Q(N)$ as well.
Any occurrence of the term $\overline{H}\xv$ can now be replaced by
a suitable linear combination of the variables $\hv$.
Write $Z_1$ for the first matrix in \eqref{eq-mat1} and 
$Z_2 = \begin{pmatrix} \hat{M} & \hat{N} & H'^T\hat{N}^TR^{-1}\hat{M} \end{pmatrix}^T$.
From equation \eqref{diag-eq} we obtain
\begin{align*}
C_{22} = \overline{H}^T \hat{N}^TR^{-1}\hat{M} \overline{H} + D_1
\end{align*}
for a diagonal matrix $D_1$ and can replace $\xv^TC_{22}\xv$ by
$\hv^T Z_3 \hv + \xv^T D_1 \xv$ with $Z_3 = \hat{N}^TR^{-1}\hat{M}$.
Then equation \eqref{eq-mat1} changes into
\begin{align*}
\yv^T Z_1 \yv + 2 \yv^T Z_2 \hv + \hv^T Z_3 \hv + \xv^T D_1 \xv = 0.
\end{align*}
Combine the parts, which only contain $\yv$ and $\hv$, into a single quadratic form
\begin{align*}
P(\yv,\hv) = \begin{pmatrix}
\yv^T & \hv^T
\end{pmatrix}
\begin{pmatrix}
Z_1 & Z_2 \\
Z_2^T & Z_3
\end{pmatrix}
\begin{pmatrix}
\yv \\ \hv
\end{pmatrix}.
\end{align*}
If we write $\lambda_i$ for the diagonal entries in $D_1$, we obtain the system
\begin{equation} \label{diag-sys}
\begin{split} 
\lambda_1x_1^2 + \ldots + \lambda_wx_w^2 &= P(\yv,\hv),  \\ 
\mu_{11}x_1 + \ldots + \mu_{1w}x_w &= h_1, \\
\vdots \quad\qquad\qquad \vdots \qquad & \ \quad \vdots \\
\mu_{d1}x_1 + \ldots + \mu_{dw}x_w &= h_d,
\end{split}
\end{equation}
where the matrix $(\mu_{ij}) \in \Z^{d \times w}$ has property [AP]
and the matrix of the quadratic form $P$ has off-rank at least $r \geq d$ 
uniformly in $\hv$.
The number of diagonal variables is at least $w \geq s-4r +2d$, which is at least $5 + 2d - r > 0$ 
for $s \geq 5 + 3r$ and $r \leq 4$.

If necessary, we can multiply the first equation \eqref{diag-sys} by a suitable non-zero integer 
to ensure that all the entries in the matrix $P$ are
integers, thereby avoiding any complications later.

\section{Preparation for Section \ref{Konvex}}

In Section \ref{PD} we found $d \leq 4$ linear equations such that by adding them to our original quadric 
$Q(\xv) = 0$ we end up with the partially diagonal form \eqref{diag-sys}. 
Some of the coefficients $\lambda_i$ can be zero and we split the vector $\xv = (x_1,\ldots,x_w)$
according to this condition. Denote by $z_i$ the variables with vanishing coefficients in the quadratic equation.
We obtain after renaming
\begin{equation} \label{Main-eq-sys}
\begin{split}  
\qquad\qquad \lambda_1x_1^2 + \ldots + \lambda_ux_u^2 &= P(\yv,\hv),  \\ 
\nu_{11}z_1 + \ldots + \nu_{1v}z_v + \mu_{11}x_1 + \ldots + \mu_{1u}x_u &= h_1, \\
\vdots \quad\qquad\qquad \vdots \qquad \quad \vdots \quad\qquad\qquad \vdots \qquad & \ \quad \vdots \\
\nu_{d1}z_1 + \ldots + \nu_{dv}z_v + \mu_{d1}x_1 + \ldots + \mu_{du}x_u &= h_d,
\end{split}
\end{equation}
where $u + v = w$. We recall that the integer variables $h_i$
can be restricted to a bounded range $[-CN,CN]$ for some $C \geq 1$.

There are now two cases to consider, which result in completely different treatments
of the system of equations.
In the first case (the one we consider in the next section) 
the columns $(\nu_{ij})_{1 \leq i \leq d}$ in the linear part of the system are linearly independent. 
In particular, this implies that $v \leq d$.
If they are not linearly independent, we deal with this system in Section \ref{Linsystem} and use 
methods of Gowers \cite{Gow} on linear equations in dense sets.

Recall the definition of $V_{g}(\alpha,\beta)$ and $L_g(\alpha)$
from \eqref{eq-Vg} and define a new exponential sum corresponding to the right hand side
of \eqref{Main-eq-sys} by
\begin{align*}
T_{g}(\alpha,\bev) = 
\sum_{\substack{\yv \leq N \\ |\hv| \leq CN}} g(\yv) e(\alpha P(\yv,\hv) + \bev \cdot \hv).
\end{align*}
In the case $g(\yv) = 1_{\mathcal{A}^{s-w}}(\yv) = 1_{\mathcal{A}}(y_1) \cdots 1_{\mathcal{A}}(y_{s-w})$
we write $T_{\mathcal{A}}(\alpha,\bev)$ instead.
We can use the trivial estimate
\begin{align*}
|T_{g}(\alpha,\bev)| \leq \sum_{|\hv| \leq CN} \Big| \sum_{\yv \leq N} g(\yv) e(\alpha P(\yv,\hv)) \Big|
\end{align*}
to remove the linear term. Since the quadratic form $P(\yv,\hv)$ has off-rank at least $r$ for any $\hv$,
we get (as in the proof of the second part of Theorem \ref{Lp-thm1})
\begin{align} \label{eq-T-bilin}
|T_g(\alpha,\bev)| \ll N^{s-w+d-r} K(\alpha)^r,
\end{align}
for $K(\alpha)$ as in \eqref{eq-K} and any bounded $g(\yv) = g_1(y_1) \cdots g_{s-w}(y_{s-w})$.

To simplify the estimation of \eqref{eq-int-absch} we introduce
abbreviations for two often occuring `actions'.

\begin{itemize}
\item[\bf Rep:] (Replace) \emph{Suppose that one is given a set of linearly independent linear forms $\Omega$ and another 
linear form $l$. Choose one linear form $l'$ from the set $\Omega$ and exchange
it with $l$ in such a way that $\{l\} \cup \Omega \backslash \{l'\}$ is still linearly independent. 
We use this action to replace one exponential sum by another in such a way that the corresponding
linear forms obey the replacement property.}
\item[\bf EaS:] (Estimate and Simplify) \emph{If we estimate an exponential sum integral by a sum of several such integrals,
we need to deal only with the maximal contribution. This happens, for example, when we apply the inequality
$x_1^{p_1}\cdots x_n^{p_n} \leq x_1^p + \ldots + x_n^p \ll x_1^p$ with $p = p_1 + \ldots + p_n$
and assume without loss of generality that $x_1^p$ is the largest term. 
If different terms require different treatments, we make an additional case analysis.}
\end{itemize}

\section{Convexity and Correlation Estimates} \label{Konvex}

Now that we set up the necessary notation and conventions, 
we can write the number of solutions to \eqref{Main-eq-sys} as the Fourier-integral
\begin{align} \label{eq-sol-int}
\int_{\T^{d+1}} T_{\mathcal{A}}(\alpha,\bev) \prod_{i=1}^v L_{\mathcal{A}}(\nuv_i \cdot \bev) 
\prod_{j=1}^u V_{\mathcal{A}}(\lambda_j\alpha,\muv_j \cdot \bev)  \,d\alpha d\bev.
\end{align}
For $\mathcal{A} \subset [1,N]$ and $\delta = N^{-1}|\mathcal{A}|$ we define the balanced function $f$ by
\begin{align} \label{balanced}
f(n) = 1_{\mathcal{A}}(n) - \delta.
\end{align}
We replace each occurrence of $1_{\mathcal{A}}$ in the integral above by $f + \delta 1_{[1,N]}$
and expand \eqref{eq-sol-int} into $2^{s}$ integrals of the form
\begin{align} \label{eq-int-absch}
E = \int_{\T^{d+1}} T_g(\alpha,\bev) \prod_{i=1}^v L_{f_i}(\nuv_i \cdot \bev)
\prod_{j=1}^u V_{g_j}(\lambda_j\alpha,\muv_j \cdot \bev) \,d\alpha d\bev,
\end{align}
where $f_i,g_j \in \{f, \delta 1_{[1,N]}\}$ and $g$ is a product of $s-w$ such functions.
We consider the $2^s -1$ integrals that contain the function $f$ as `error terms' and give upper bounds
on their size to deduce correlation estimates for the exponential sums later.

We start with the case, where $f_1 = f$ and show later how to modify the argument to obtain the 
estimate in the other cases. By pulling out half of $L_f$, we obtain from \eqref{eq-int-absch}
the estimate
\begin{align*}
E \ll \sup_{\beta} |L_{f}(\beta)|^{1/2} \int_{\T^{d+1}} |T_g| 
|L_{f_1}|^{1/2} \prod_{i=2}^v |L_{f_i}|
\prod_{j=1}^u |V_{g_j}| \,d\alpha d\bev,
\end{align*}
where we left out the dependences on the variables to save space.
By property [AP] from Section \ref{PD}, the linear forms $\nuv_i \cdot \bev$ and $\muv_j \cdot \bev$ 
in the exponential sums $L_{f_i}$ and $V_{g_j}$
contain two basis sets after removing $\nuv_1 \cdot \bev$.
We can group these exponential sums into two products $W_1$ and $W_2$ corresponding to the two bases.
Since the linear forms $\nuv_i \cdot \bev$ are linearly independent by assumption, 
we can make sure that all $L_{f_i}$ for $2 \leq i \leq u$
are contained in $W_1$. We obtain (after renaming) an estimate of the form
\begin{align} \label{eq-10}
E \ll \sup_{\beta} |L_{f}(\beta)|^{1/2} \int_{\T^{d+1}} |T_g| 
|L_{f_1}|^{1/2} |W_1||W_2| \prod_{j=1}^{w-2d-1} |V_{g_j}| \,d\alpha d\bev.
\end{align}
Since the half-power of a linear exponential sum would cause problems later, we find by [Rep]
an exponential sum $V$ inside $W_2$ that can be replaced by $L_{f_1}$.
Then we apply [EaS] with the estimate $|L|^{1/2} |V| \leq |V|^{1/2} |L| + |V|^{3/2}$ 
to replace $|L|^{1/2}$ by $|V|^{1/2}$.
In the case of the first summand, we have substituted $|V|$ by $|L|$ inside $W_2$.
Assume without loss of generality that $V$ is $V_{g_{w-2d}}$. Our estimate is now
\begin{align*}
E \ll \sup_{\beta} |L_{f}(\beta)|^{1/2} \int_{\T^{d+1}} |T_g| 
|V_{g_{w-2d}}|^{1/2} |W_1||W_2| \prod_{j=1}^{w-2d-1} |V_{g_j}| \,d\alpha d\bev.
\end{align*}

At this point we are in the position to apply a cascade of estimates.
We replace $|T_g|$ by the right hand side of \eqref{eq-T-bilin}, use [EaS] to replace the product of the
$V_{g_j}$ (including $|V_{g_{w-2d}}|^{1/2}$) by $|V_{g_1}|^{w-2d-1/2}$ and again [EaS] with the estimate 
$|W_1W_2| \leq |W_1|^2 + |W_2|^2$.
We obtain a much simpler upper bound of the form
\begin{align*}
E \ll N^{s-w-r+d} \sup_{\beta} |L_{f}(\beta)|^{1/2} \int_{\T^{d+1}} K^r |V_{g_1}|^{w-2d-1/2} |W_1|^2 \,d\alpha d\bev.
\end{align*}
Before we proceed further, we reduce the power of $|V_{g_1}|$ in order to obtain a more uniform estimate in 
$\delta$ later on. We can replace $w-2d \geq s-4r \geq 5-r$ by the minimal value $5 - r$ for which 
the following argument works and pull out any
additional powers of $|V_{g_1}|$. With the trivial estimate $|V_{g_1}| \leq 2\delta N$ we obtain
\begin{align*}
E \ll \delta^{w-2d+r-5} N^{s-d-5} \sup_{\beta} |L_{f}(\beta)|^{1/2} \int_{\T^{d+1}} K^r |V_{g_1}|^{9/2-r} |W_1|^2 \,d\alpha d\bev.
\end{align*}

Now we continue with the main argument. The estimate 
\begin{align} \label{KV-est}
K^r |V_{g_1}|^{9/2-r} \leq K^{9/2} + |V_{g_1}|^{9/2} 
\end{align}
results in two cases. If the first expression on the right hand side gives the dominating contribution, then
\begin{align*}
E \ll \delta^{w-2d+r-5} N^{s-d-5} \sup_{\beta} |L_{f}(\beta)|^{1/2} \int_{\T^{d+1}} K^{9/2} |W_1|^2 \,d\alpha d\bev.
\end{align*}
Here we can use the fact that $K$ only depends on $\alpha$ with Parseval's identity
(or equivalently `interpretation of the integral as a diophantine equation')
\begin{align*}
\int_{\T^{d}}  |W_1(\alpha,\bev)|^2 \, d\bev = N^d,
\end{align*}
to get the bound
\begin{align*}
E \ll \delta^{w-2d+r-5} N^{s-5} \sup_{\beta} |L_{f}(\beta)|^{1/2} \int_{\T} K(\alpha)^{9/2}  \,d\alpha.
\end{align*}
We can apply Lemma \ref{K-bound} to get the final upper bound
\begin{align} \label{eq-up1}
E \ll \delta^{w-2d+r-5} N^{s-5/2} \sup_{\beta} |L_{f}(\beta)|^{1/2}.
\end{align}

If the second term in \eqref{KV-est} gives the dominating contribution, then
\begin{align} \label{eq-LV}
E \ll \delta^{w-2d+r-5} N^{s-d-5} \sup_{\beta} |L_{f}(\beta)|^{1/2} \int_{\T^{d+1}} |V_{g_1}|^{9/2} |W_1|^2 \,d\alpha d\bev,
\end{align}
and things are slightly more complicated. First we find by [Rep] the linear form in $W_1$ that corresponds to
$\muv_1 \cdot \bev$. This linear form can either belong to another function $V_{g_j}$
or $\muv_1 \cdot \bev$ must be in the span of the linear forms of the 
linear exponential sums $L_{f_i}$ that appear in $W_1$. We have to deal with these subcases differently.

In the first subcase, we can use [EaS] with the estimate
\begin{align*}
|V_{g_1}|^{9/2}|V_{g_j}|^2 \ll |V_{g_1}|^{13/2}
\end{align*} 
and make the change of variables $\gamma_i = \nuv_i \cdot \bev$ and $\gamma_j = \muv_j \cdot \bev$ 
to bound the integral in \eqref{eq-LV} by a multiple of
\begin{align*}
\int_{\T^{d+1}} |V_{g_1}(\lambda_1\alpha, \gamma_1)|^{13/2} \prod_{i=2}^h |L_{f_i}(\gamma_i)|^2
\prod_{j=h+1}^d |V_{g_j}(\lambda_j\alpha,\gamma_j)|^2 \,d\alpha d\gamv,
\end{align*}
for some $h \leq d$ dependent on the previous steps. 
Since the variables $\gamma_j$ appear now separately, we can use Parseval's identity $d-1$ times
to integrate out $\gamma_2$ to $\gamma_d$ and obtain
\begin{align*}
E \ll \delta^{w-2d+r-5} N^{s-6} \sup_{\beta} |L_{f}(\beta)|^{1/2} 
\int_{\T^2} |V_{g_1}(\lambda_1\alpha, \gamma_1)|^{13/2} \,d\alpha d\gamma_1.
\end{align*}
A change of variables (to remove $\lambda_1$) and Theorem \ref{VA-est} imply
the final bound
\begin{align*}
E \ll \delta^{w-2d+r-5} N^{s-5/2}\sup_{\beta} |L_{f}(\beta)|^{1/2}.
\end{align*}

In the second subcase of the analysis of \eqref{eq-LV}, we can bound the integral
in \eqref{eq-LV} (after a change of variables) by a multiple of
\begin{align*}
\int_{\T^{d+1}} |V_{g_1}(\lambda_1\alpha, \nuv \cdot \gamv)|^{9/2} 
\prod_{i=1}^h |L_{f_i}(\gamma_i)|^2
\prod_{j=h+1}^d |V_{g_j}(\lambda_j\alpha,\gamma_j)|^2 \,d\alpha d\gamv,
\end{align*}
for some $\nuv \in \Q^{d}$ with the property that the linear form 
$\nuv \cdot \gamv$ only contains the variables $\gamma_1,\ldots,\gamma_h$.
As before, we integrate out the functions $|V_{g_j}|^2$ by Parseval's identity and end up with
\begin{align*}
N^{d-h} \int_{\T^{h+1}} |V_{g_1}(\lambda_1\alpha, \nuv \cdot \gamv)|^{9/2} 
\prod_{i=1}^h |L_{f_i}(\gamma_i)|^2 \,d\alpha d\gamma_1\cdots d\gamma_h.
\end{align*}
Now we are again in a situation similar to the case with $K(\alpha)$.
Observe that the functions $L_{f_i}$ are independent of $\alpha$.
We can use Lemma \ref{Vg4-est} and Parseval's identity for the integrals 
over $\gamma_1,\ldots,\gamma_h$ to bound this from above by
\begin{align*}
 N^{d-h} \sup_{\beta} \int_{\T} & |V_{g_1}(\lambda_1\alpha, \beta)|^{9/2} \, d\alpha
\int_{\T^{h}} \prod_{i=1}^h |L_{f_i}(\gamma_i)|^2 \,d\gamma_1\cdots d\gamma_h \\
 \ll & N^{d-h} N^{5/2} N^{h} = N^{d+5/2}.
\end{align*}
Combined with \eqref{eq-LV}, we obtain estimate \eqref{eq-up1} again.

This is the end of the estimates under the assumption that $f_1 = f$. We now address the other cases.
Obviously, the same procedure works, when $f_i = f$ for some other $i \leq v$.
If $g_j = f$ for some $j \leq u$, the argument is even simpler. 
We obtain instead of \eqref{eq-10} the inequality
\begin{align*}
E \ll \sup_{\alpha,\beta} |V_{f}(\alpha,\beta)|^{1/2} \int_{\T^{d+1}} |T_g| 
|V_{g_{w-2d}}|^{1/2} |W_1||W_2| \prod_{j=1}^{w-2d-1} |V_{g_j}| \,d\alpha d\bev,
\end{align*}
which can be treated in the same manner as before and gives the final bound
\begin{align*}
E \ll \delta^{w-2d+r-5} N^{s-5/2} \sup_{\alpha,\beta} |V_{f}(\alpha,\beta)|^{1/2}.
\end{align*}

To describe the last remaining case write $g(\yv) = \prod_{k=1}^{s-w} h_k(y_k)$
for the function $g$ appearing in $T_g$ with $h_k = f$ for some $1 \leq k \leq s-w$.
Instead of pulling out half of $T_g$, we take the $1/(2r)$-th power.
This gives us
\begin{align*}
E \ll \sup_{\alpha,\bev} |T_g(\alpha,\bev)|^{1/2r} \int_{\T^{d+1}} |T_g|^{1-1/2r} 
|W_1||W_2| \prod_{j=1}^{w-2d} |V_{g_j}| \,d\alpha d\bev
\end{align*}
instead of \eqref{eq-10}.
The same estimates as before apply, where $K(\alpha)^{1/2}$ is replaced by
$|V_{g_1}|^{1/2}$. This does not matter since we use [EaS] with \eqref{KV-est} to separate the cases
with $K$ and $V_{g_1}$ anyway. The final estimate has the slightly different form
\begin{align*}
E \ll \delta^{w-2d+r-5} N^{s-2 - (s-w+d)/2r} \sup_{\alpha,\bev} |T_g(\alpha,\bev)|^{1/2r}.
\end{align*}
This concludes the last case of the error term estimates and we proceed to deduce the
corelation estimates.

One of the integrals appearing in the expansion of \eqref{eq-sol-int} is equal to
\begin{align} \label{eq-expt-sol}
\delta^s \int_{\T^{d+1}} T(\alpha,\bev) \prod_{i=1}^v L(\nuv_i \cdot \bev) 
\prod_{j=1}^u V(\lambda_j\alpha,\muv_j \cdot \bev)  \,d\alpha d\bev
\end{align}
and gives the expected number of solutions to \eqref{Main-eq-sys} with $x_i,z_j \in [1,N]$.
From the discussion in Section \ref{Not-Out} and Lemma \ref{Standard-quad-result}, we know that \eqref{eq-expt-sol} is
bounded from below by $\gg \delta^s N^{s-2}$. 
On the other hand, \eqref{eq-sol-int} is of size exactly $\delta N$ since by assumption there are only trivial solutions
to our system \eqref{Main-eq-sys}.
Therefore, at least one of the $2^s-1$ `error terms' $E$ has to be of size 
$\gg \delta^s N^{s-2}$ if $N \gg_Q \delta^{-2}$.
The three previously obtained upper bounds transform 
with $w \geq s-4r+2d$ into one of the correlation estimates
\begin{align*}
\sup_{\alpha, \bev} |T_{g}| \gg \delta^{2r(3r+5)} N^{s-w+d}, \quad
\sup_{\alpha, \beta} |V_{f}| \gg \delta^{6r+10} N \mbox{\ \ or \ \ }
\sup_{\alpha} |L_{f}| \gg \delta^{6r+10} N,
\end{align*}
as long as $N \gg_Q \delta^{-2}$ and $r \leq 4$. In the next section, 
these lower bounds are used to obtain structural information about the set $\mathcal{A}$.

\section{Density Increment Method} \label{Density}

In equation \eqref{CorS-est} and Section \ref{Konvex} we have found the correlation estimates
\begin{align*}
& \sup_{\alpha} |S_{f_i}(\alpha)| \gg \delta^{s+80} N^s, \qquad 
\sup_{\alpha, \bev} |T_{g}(\alpha,\bev)| \gg \delta^{136} N^{s-w+d}, \\
& \sup_{\alpha} |L_{f}(\alpha)| \gg \delta^{34} N, \qquad \quad\
\sup_{\alpha, \beta} |V_{f}(\alpha,\beta)| \gg \delta^{34} N,
\end{align*}
that hold for suitable functions $f_i,g$ as long as $N \gg_Q \delta^{-2}$.
To reduce our work we transform the first three estimates to the fourth (with $\delta^{136}$ instead of $\delta^{34}$)
in the following way. We have from \eqref{fi-def} the representation
\begin{align*}
S_{f_i}(\alpha) = \sum_{\xv \leq N}(1_{\mathcal{A}}(x_i) - \delta) 
\delta^{i-1}\prod_{j > i} 1_{\mathcal{A}}(x_j) e(\alpha Q(\xv))\\
= \sum_{\xv' \leq N} \delta^{i-1}\prod_{j > i} 1_{\mathcal{A}}(x_j) \sum_{x_i \leq N} f(x_i) e(\alpha q_{\xv'}(x_i)),
\end{align*}
where $\xv'$ is the vector of variables without $x_i$ and $q_{\xv'}(x_i) = Q(\xv)$ is seen
as a quadratic polynomial in $x_i$
with linear and constant coefficients depending on $\xv'$.
Therefore, we deduce for some $d \in \Z$ and $\beta = \beta(\xv',\alpha)$ the estimate
\begin{align*}
|S_{f_i}(\alpha)| \leq  \sum_{\xv' \leq N} \delta^{i-1}\prod_{j > i} 1_{\mathcal{A}}(x_j) |V_f(d\alpha,\beta)|.
\end{align*}
Take the supremum over $\alpha$ and $\beta$ of $V$ outside the sum and conclude that
\begin{align*}
 \delta^{s+80} N^s \ll \sup_{\alpha} |S_{f_i}(\alpha)| 
\leq & \sup_{\alpha,\beta} |V_f(\alpha,\beta)| \sum_{\xv' \leq N} \delta^{i-1}\prod_{j > i} 1_{\mathcal{A}}(x_j) \\
\leq & \delta^{s-1} N^{s-1} \sup_{\alpha, \beta} |V_f(\alpha,\beta)|,
\end{align*} 
which gives the desired implication.

In a similar way we may reduce the second estimate to the fourth. 
The function $g$ in $T_g$ is a product of functions 
$f$ from \eqref{balanced} and $\delta 1_{[1,N]}$. 
We expand all but one of the balanced functions $f$ in $g$ into $1_{\mathcal{A}}-\delta$, 
giving us several exponential sums of the form
\begin{align*}
T_{g_i}(\alpha,\bev) = 
\sum_{\yv \leq N, |\hv| \leq CN} f(y_i) 
\prod_{j \neq i} g_j(y_j) e(\alpha P(\yv,\hv) + \bev \cdot \hv)
\end{align*}
with $g_j \in \{\delta 1_{[1,N]},1_{\mathcal{A}}\}$.
At least one of these has to be big and this implies a 
correlation estimate in exactly the same way as for $S_{f_i}$
if we estimate the sums over $\hv$ trivially.
Finally, we observe the identity $L_f(\beta) = V_f(0,\beta)$ which reduces the third case to the fourth.

The fourth correlation can be used to obtain a density increment by the following lemma.

\begin{lem} \label{exp-prog}
If $|V_f(\alpha,\beta)| \geq \eta N$ for some $(\alpha,\beta) \in \T^2$ and $\eta > 0$, then there is
an arithmetic progression $P \subset [1,N]$ of length $|P| \gg \eta^{2} N^{1/16}$ with
\begin{align*}
|\mathcal{A} \cap P| \geq (\delta + \eta/4) |P|.
\end{align*}
\end{lem}

\begin{proof}
This is the Lemma B.1 in Appendix B of \cite{Keil}. 
Similar statements for finite fields can be found in \cite{Gow}, for example.
\end{proof}

We use the large Fourier estimate $|V_f(\alpha,\beta)| \gg \delta^{136}N$ from above with Lemma \ref{exp-prog} 
to find a progression $P$ of length $\gg \delta^{272} N^{1/16}$, such that $\mathcal{A}$ has density
$\geq \delta + \lambda\delta^{136}$ on $P$, where $\lambda > 0$ is an absolute constant.
Due to the translation and dilation invariance of $Q(\xv)=0$, we end up with the same
problem on a subprogression, but with a slightly higher density.

Since the density is bounded by one, this procedure cannot last more than $\lambda^{-1}\delta^{-136}$ steps 
before reaching a contradiction.
This means that at some stage we have a non-trivial solution or the size of our progression is
$\ll \delta^{-2}$. The first option is not available by assumption.
Therefore, we have 
\begin{align*}
\delta^{150}N^{(1/16)^{\lambda^{-1}\delta^{-136}}} \ll_Q \delta^{-2}.
\end{align*}
Rearranging for $\delta$ we can deduce that $\delta \ll (\log\log N)^{-c}$ with $c = (137)^{-1}$, for example.

\section{A Linear Subsystem} \label{Linsystem}

In this section we consider the last two cases of Theorem \ref{Main-thm}.
First let us look at the case from Section \ref{Konvex}, 
when the columns of the linear part in \eqref{Main-eq-sys} turn out to be
linearly dependent. This allows us to extract a linear subsystem as follows.

Since the quadric is translation invariant we can set $y_k = z_j = x_i = z_0 \in \mathcal{A}$
for any $z_0$, and get a trivial solution of \eqref{Main-eq-sys} 
with uniquely defined $h_i(z_0) = -\nu_{i0} z_0$ for some $\nu_{i0} \in \Z$. 
Equipped with this information, we can set 
$y_k = x_i = z_0 \in \mathcal{A}$ and $h_i = -\nu_0 z_0$ for all $k$ and $i$ in \eqref{Main-eq-sys}
and obtain the translation invariant system
\begin{equation} \label{eq-linsys}
\begin{split}  
\nu_{11}z_1 + \ldots + \nu_{1v}z_v + \nu_{10}z_0 &= 0, \\
\vdots \qquad \qquad \qquad \qquad \quad \vdots \quad  &\ \quad  \vdots \\
\nu_{d1}z_1 + \ldots + \nu_{dv}z_v + \nu_{d0}z_0 &= 0,
\end{split}
\end{equation}
with coefficients $\nu_{i0} \in \Z$.
By assumption, the rank of the first $v$ columns is at most $v-1$. 
By translation invariance, the last column is a linear combination of the first $v$
columns and this implies that the rank of the whole coefficient matrix is still at most $v-1$.
This gives a linear system in $v+1$ variables with at most $v-1$
independent equations.
A very similar system appears in Section \ref{Not-Out} in the situation, where
the number of non-zero coefficients in \eqref{diagonal-equation} is at most four.
Both systems can be treated by the method of Gowers \cite{Gow}, as we show now.

First we remove equations from \eqref{eq-linsys} until
the remaining rows become linearly independent.
To simplify notation, we can assume that we are left with $w$ equations
and $w+2$ variables. If more variables are left over, we can `fuse' them
by setting them equal to $z_0$.

We can also assume that the columns of $E$ are in general position. To see this
we take $\Z$-linear combinations of rows and rename variables to transform 
the equation $E \zv = 0$ into the form
\begin{align} \label{eq-linmodified}
\begin{pmatrix}
D & \av & \bv
\end{pmatrix} \cdot \zv = 0,
\end{align}
where $D \in \Z^{w \times w}$ is a diagonal matrix of full rank with non-zero diagonal entries and 
$\av,\bv \in \Z^w$ are integer vectors.
The property that the columns of the matrix in \eqref{eq-linmodified} are in general position 
translates into the arithmetic conditions that 
$a_i \neq 0, b_i \neq 0$ for all $i$ and $a_ib_j - a_jb_i \neq 0$ for $i \neq j$.

If this is not the case, we can combine two rows to deduce an equation of the form 
$m z_i = m z_j$ for some $1 \leq i < j \leq w+2$ and $m \neq 0$.
(The form of this equation is forced by translation invariance.)
It is then possible to reduce the system by one equation and one variable, leaving
us with the same situation with parameter $w-1$ instead of $w$.

We should briefly discuss the case $w = 1$. In this case we have one equation in three
variables of the form
\begin{align*}
d_1x_1 + d_2x_2 + d_3x_3 = 0,
\end{align*}
with $d_1 + d_2 + d_3 = 0$. If $d_i \neq 0$ for all $1 \leq i \leq 3$, then we are in the situation
of Roth's paper \cite{Roth} and we obtain a bound of the form $O((\log\log N)^{-1})$ for the density of $\mathcal{A}$.
If we have $d_i = 0$ for one of the coefficients, then we directly get a non-trivial solution in $\mathcal{A}$
as long as $|\mathcal{A}| \geq 2$.

Now consider the difference
\begin{align*}
\sum_{\zv, E\zv = 0} 1_{\mathcal{A}}(z_0) \cdots 1_{\mathcal{A}}(z_{w+1}) 
- \sum_{\zv, E\zv = 0} \delta 1_{[1,N]}(z_0) \cdots \delta 1_{[1,N]}(z_{w+1}).
\end{align*}
Since we assume that our system has only trivial solutions, this difference is
either of size around $\delta^{w+2}N^{2}$ or we have $N \ll \delta^{-w-1}$.
In the second case, we are done. Otherwise, we can write 
$1_{\mathcal{A}} = f + \delta 1_{[1,N]}$ for the balanced function $f$ from \eqref{balanced}
and expand the first sum into $2^{w+2}$ terms, one of which is cancelled by the second term.
Then we bound the remaining sums of the form 
\begin{align} \label{eq-norm-est}
\sum_{\zv, E\zv = 0} f_0(z_0) \cdots f_{w+1}(z_{w+1}).
\end{align}
with functions $f_i \in \{f, \delta 1_{[1,N]}\}$ and at least one of these $f_i$ equal to $f$.

To simplify the exposition and be able to cite a result from \cite{Gow}, we convert this sum into one with variables 
$z_i \in \Z/M\Z$ for some prime $M$ dependent on $N$ and $E$.
We choose the prime $M$ in an interval $[C_EN,2C_EN]$, which is possible by Bertrand's postulate.
More precisely, we take $C_E$ in such a way that the equation $E \zv = 0$ in $\Z/M\Z$
with $z_i \in [1,N]$ implies that $E\zv=0$ in $\Z$. 
This is always possible, if $C_E$ is big enough to avoid `wrap-around issues'. 
We set the functions $f_i(z_i)=0$ for $z_i \notin [1,N]$.

The `error terms' in \eqref{eq-norm-est} can be bounded by more symmetric expressions, the
so called {\em uniformity norms}. This is done in Appendix \ref{AppB}, where we deduce 
Lemma \ref{Uw-est}. We get the estimate (see Appendix \ref{AppB} for a definition of $\Delta$)
\begin{equation}  \label{eq-Gow-est}
\begin{split} 
& M^{-2} \sum_{\zv, E\zv = 0} f_0(z_0) \cdots f_{w+1}(z_{w+1}) \\
& \quad \leq \Big(M^{-w-2} \sum_{h_1,\ldots,h_w} \Big| \sum_{z_{w+1}} 
\Delta(f_{w+1};\hv)(z_{w+1})\Big|^2 \Big)^{2^{-w-1}}.
\end{split}
\end{equation}
For this to be useful, we have to assume that $f_{w+1} = f$. If this was not the case, we
can rename variables 
to obtain the desired outcome.

According to \cite{Gow} we define a function $f: \Z/M\Z \to \C$ to be $\alpha$-uniform of degree $w$ if 
\begin{align*}
\sum_{h_1,\ldots,h_w} \Big| \sum_{z} \Delta(f;\hv)(z)\Big|^2 \leq \alpha M^{w+2}.
\end{align*}
A set $\mathcal{A}$ is $\alpha$-uniform if its balanced function $f$ given by \eqref{balanced} is $\alpha$-uniform.
By assumption, our error term is $\gg \delta^{w+2}N^{2} \gg \delta^{w+2}M^{2}$, so in combination with the estimate \eqref{eq-Gow-est}, 
we derive that $f$ is not $\mu (\delta^{w+2})^{2^{w+1}}$-uniform for some small $\mu > 0$.
Combined with the following theorem of Gowers this allows us to deduce 
a density increment for $\mathcal{A}$ on a subprogression, similar to the one
in Lemma \ref{exp-prog}.

\begin{thm} \label{Gow-thm}
Let $\alpha \leq 1/2$ and let $\mathcal{A} \subset \Z/M\Z$ be a set which fails to be 
$\alpha$-unform of degree $k$. There exists a partition of $\Z/M\Z$ into arithmetic progressions
$P_1,\ldots, P_K$ of average size $M^{\alpha^{2^{2^{k+10}}}}$ such that
\begin{align*}
\sum_{j =1}^K \Big|\sum_{s \in P_j} f(s) \Big| \geq \alpha^{2^{2^{k+10}}} M.
\end{align*}
\end{thm}

\begin{proof}
This is Theorem 18.5 from \cite{Gow}.
\end{proof}

We note that the term \emph{arithmetic progression} in the theorem refers to `proper' 
arithmetic progressions, which keep their structure if projected down to $[1,N]$.
For details the reader is referred to \cite{Gow}.

To obtain the density increment we observe that 
\begin{align*}
\sum_{j =1}^K \sum_{s \in P_j} f(s) = \sum_{s \leq N} f(s) =  0
\end{align*}
by the definition of $f$ given in \eqref{balanced}. 
Therefore, by Theorem \ref{Gow-thm}, we have
\begin{align*}
\sum_{s \in P_j} f(s) \geq \frac{1}{2} \alpha^{2^{2^{k+10}}} |P_j|
\end{align*}
for some $j \leq K$, which can be translated into a density increment of size
\begin{align*}
\frac{1}{2} \Big(\mu (\delta^{w+2})^{2^{w+1}}\Big)^{2^{2^{w+10}}}
\end{align*}
on $P_j$. If we perform the same iteration as in Section \ref{Density},
we obtain the same result but with $c=2^{-2^{15}}$. This is due to the fact that
we need the argument only for $w \leq 4$,
which makes $c$ independent of the number of variables $s$.

\appendix

\section{A Version of Vinogradov's Lemma}   \label{AppA}

Here we prove Lemma \ref{Weyl-est}, a version of Vinogradov's lemma, 
that we need for the estimation of bilinear exponential sums.
While there are many versions in the literature, none of those seems to be good enough for our purpose.
Since we need to reprove it with explicit dependance on $\beta$, we take the opportunity to
state it with explicit constants as well, to provide a reference for possible numerical applications.

\begin{lem} \label{Weyl-est}
Let $\alpha = a/q + \beta$ with $|\beta| \leq \frac{1}{qN}$ and $(a;q) = 1$, then we have 
\begin{align*}
\sum_{|h| \leq N} \min\{N, \|\alpha h\|^{-1}\} & 
\leq 6N + 6 \min\left\{\frac{N^2}{q}, \frac{1}{|\beta| q} \left(|\log(|\beta| N^2)| + 2\right)\right\}\\
& + (4N+2q) \cdot (1 + \log q).
\end{align*}
\end{lem}

\begin{proof}
We insert the formula $\alpha = a/q + \beta$ and rearrange the expression 
according to the residue class of $h \mod q$, giving us
\begin{align*}
& \sum_{|h| \leq N} \min\{N, \|ah/q + \beta h\|^{-1}\}
=& \sum_{c= 0}^{q-1} \sum_{\substack{|h| \leq N\\ h \equiv c \mod q}} \min\{N, \|ac/q + \beta h\|^{-1}\}.
\end{align*}

Since $|\beta| \leq \frac{1}{qN}$ and $h \leq N$, the term $\beta h$ is bounded by $1/q$ 
and since $(a;q) = 1$, there are at most three values of $c$ such that $\|ac/q + \beta h\| < 1/q$ is possible.
For the other values of $c$, we can bound the expression $\|ac/q + \beta h\|$ from below
by $d/q$ for some $1 \leq d \leq q/2$ independent of $|h| \leq N$.
Take the maximal $d$ for which the inequality holds and arrange the values of $c$ according
to these $d$-values. There are at most two values of $c$ for each $d$ and we end up with the bound
\begin{align*}
\sum_{|h| \leq N} \min\{N, \|\alpha h\|^{-1}\} & \leq  3 \sup_{0 \leq c \leq q-1}\ \sup_{x_0 \in [0,1]} 
\sum_{\substack{|h| \leq N\\ h \equiv c \mod q}} \min\{N, \|x_0 + \beta h\|^{-1}\} \\
& + 2 \cdot \frac{2N+q}{q} \sum_{1 \leq d \leq q/2}  \|d/q\|^{-1}.
\end{align*}
We use the estimate
$\sum\limits_{1 \leq d \leq q/2} d^{-1} \leq 1+ \log q$ and get for the second term
\begin{align*}
2 \cdot \frac{2N+q}{q} \sum_{1 \leq d \leq q/2}  \|d/q\|^{-1} \leq 2 \cdot (2N+q) \cdot (1 + \log q),
\end{align*}
which accounts for the second line in the estimate of Lemma \ref{Weyl-est}.
Write $h = c + lq$ such that the first expression changes into
\begin{align*}
3 \sup_{0 \leq c \leq q-1}\ \sup_{x_0 \in [0,1]} \sum_{|l+c/q| \leq N/q} \min\{N, \|x_0 + \beta ql\|^{-1}\}.
\end{align*}
Since the function $\min\{N,\|\gamma\|^{-1}\}$ is monotone for $0 \leq \gamma \leq 1/2$ and 
symmetric around the origin, we can choose $x_0 = c= 0$ by taking care of a possible boundary term
and obtain the upper bound 
\begin{align} \label{eq-l-sum}
3N + 3\sum_{|l| \leq N/q} \min\{N, \|\beta ql\|^{-1}\}.
\end{align}
There is a small problem with this argument for $q = 1$ due to 'wrap-around issues' in $\R/\Z$, but
the bound in \eqref{eq-l-sum} is trivial in this case.

We assume without loss of generality that $\beta > 0$ and
split the summation into $-1/(\beta qN) < l < 1/(\beta qN)$ 
and the positive and negative part of the sum over $1/(\beta qN) \leq |l| \leq N/q$.
The first sum gives a contribution of at most $N \cdot (2/(\beta qN) + 1)$ and the the two other sums give
\begin{align*}
\frac{2}{\beta q} \sum_{1/(\beta qN) \leq l \leq N/q} l^{-1} 
& \leq \frac{2}{\beta q} \left( |\log(N/q) - \log(1/(\beta qN))| + 1\right)\\
& = \frac{2}{\beta q} \left(|\log(\beta N^2)| + 1\right).
\end{align*}
There is also the trivial bound $N(2N/q+1)$ for the sum in \eqref{eq-l-sum}, 
which is superior for $\beta \leq 1/N^2$.
We put everything together and replace $\beta$ by $|\beta|$ to obtain the result.
\end{proof}

\section{Uniformity Estimates} \label{AppB}

Here we want to give a proof for the uniformity estimate \eqref{eq-Gow-est}.
This is a slightly different version of a result of Gowers \cite{Gow}
and can be deduced from the very general estimates in Appendix C of \cite{GrTao3}.
Since our result is a very special case with a simpler proof, we give it here
for completeness.

Define $\Delta(f;h_1,\ldots,h_d)$ inductively by
$\Delta(f;h)(x) = f(x)\overline{f(x-h)}$ and  
\begin{align*}
\Delta(f;h_1,\ldots,h_{d+1}) = \Delta(\Delta(f;h_1,\ldots,h_d);h_{d+1}).
\end{align*}
Consider
\begin{align} \label{Def-A-sum}
A(M,E,\fv) := & M^{-2} \sum_{\xv, E\xv = 0} f_1(x_1) \cdots f_{w+2}(x_{w+2}),
\end{align}
where the variables are summed over $\Z/M\Z$ for some prime $M$.

\begin{lem} \label{Uw-est}
For a matrix $E \in \Z^{w \times (w+2)}$ with columns in general position
over $\Z/M\Z$ and bounded functions $|f_i| \leq 1$, we have the bound
\begin{align*}
|A(M,E,\fv)| \leq \Big( M^{-w-2}\sum_{h_1,\ldots,h_w} \Big| \sum_{x} \Delta(f_{w+2};\hv)(x)\Big|^2 \Big)^{1/2^{w+1}}.
\end{align*}
\end{lem}

\begin{rem}
There is nothing special about $f_{w+2}$ here except for simplicity of notation.
\end{rem}

To simplify the exposition of the proof, we don't write down all the normalisation constants $M^{-t}$,
where $t$ is the number of free variables in the summation. 
We write $\leq_M$ instead of $\leq$ to say that the estimate holds
up to a power of $M$ that has the right order of magnitude.

\begin{proof}
Since $|f_1| \leq 1$ we can estimate $A(M,E,\fv)$ by
\begin{align*}
A(M,E,\fv) & \leq_M \sum_{x_1} \Big| \sum_{E\xv = 0}^* f_2(x_2) \cdots f_{w+2}(x_{w+2}) \Big| \\
& \leq_M \Big(\sum_{x_1} \Big|\sum_{E\xv = 0}^* f_2(x_2) \cdots f_{w+2}(x_{w+2}) \Big|^2 \Big)^{1/2},
\end{align*}
where the star $*$ indicates that the inner sums run only over $x_2,\ldots,x_s$.
With $y_1=x_1$ we can write the summation inside the square root as
\begin{align*}
\sum_{x_1} 
\sum_{E\xv = 0}^* \sum_{E\yv = 0}^* f_2(x_2) \cdots f_{w+2}(x_{w+2})
\overline{f_2(y_2) \cdots f_{w+2}(y_{w+2})}.
\end{align*}
Define $h_1 := x_{w+2} - y_{w+2}$ and observe that from $E\xv=0$ and $E\yv = 0$
we get $E(\xv-\yv) = 0$. By the definition of $h_1$ we obtain that 
\begin{align*}
E \cdot (0,x_2-y_2,\ldots,x_{w+1}-y_{w+1},h_1)^T = \nullv.
\end{align*}
This equation uniquely determines the differences $x_i-y_i$ once $h_1$ is given
and, therefore, $y_i$ in terms of $x_i$ and $h_1$.
Denote this unique solution by $y_{x_i,h_1}$ and set $F_i(h_1;x_i) = f_i(x_i)\overline{f_i(y_{x_i,h_1})}$. 
Then we can write the sum in the form
\begin{align*}
\sum_{h_1} \sum_{\xv, E\xv = 0}  F_2(h_1;x_2) \cdots F_{w+1}(h_1;x_{w+1}) \Delta(f_{w+2};h_1)(x_{w+2}).
\end{align*}
Since this is just an average of another version of the original sum \eqref{Def-A-sum}
with $f_1$ removed, we can use the estimate
\begin{align*}
\sum_{h_1} \Big( \sum ...\Big) \leq_M \Big(\sum_{h_1} \Big| \sum ...\Big|^2 \Big)^{1/2}
\end{align*}
to apply the above procedure inductively to the inner sum.
After $w$ steps we have removed all functions but $\Delta(f_{w+2};h_1,\ldots,h_{w})$
and some function $F_{\hv}(x_{w+1})$. We are left with
\begin{align*}
A(M,E,\fv) \leq_M \Big( \sum_{h_1,\ldots,h_{w}} 
\sum_{\xv, E\xv = 0} F_{\hv}(x_{w+1}) \Delta(f_{w+2};\hv)(x_{w+2}) \Big)^{1/2^w}.
\end{align*}
Now we evaluate the sum over the variables $x_1,\ldots,x_w$, since their values are given once we know
the values of $x_{w+1}$ and $x_{w+2}$.
Having done this, we may rearrange the summation a last time and proceed by another 
application of the Cauchy-Schwarz-inequality and $|F_{\hv}(x_{w+1})| \leq 1$ to obtain
\begin{align*}
A(M,E,\fv) & \leq_M \Bigg( \sum_{h_1,\ldots,h_{w}} 
\Big( \sum_{x_{w+1}} 1 \Big) \Big| \sum_{x_{w+2}} \Delta(f_{w+2};\hv)(x_{w+2}) \Big| \Bigg)^{1/2^w}\\
& \leq_M 
\Big( \sum_{h_1,\ldots,h_w} \Big| \sum_{x_{w+2}} \Delta(f_{w+2};\hv)(x_{w+2})\Big|^2 \Big)^{1/2^{w+1}},
\end{align*}
which is the estimate in Lemma \ref{Uw-est} up to normalisation.
\end{proof}

\section{$L^p$-Estimates for Quadratic Exponential Sums}  \label{AppC}

In our recent work \cite{Keil}, we reproved a result of Bourgain \cite{Bour2}, 
namely the following estimate for the two-dimensional quadratic exponential sum.

\begin{thm} \label{VA-est}
Let $V_g(\alpha,\beta)$ be defined as in \eqref{eq-Vg} for a function $g$ with $|g(n)| \leq 1$. 
Then for $p > 6$, we have
\begin{align*}
\int_{\T^2} |V_g(\alpha,\beta)|^p \,d\alpha d\beta \ll_p N^{p-3}.
\end{align*}
\end{thm}

\begin{proof}
This is Theorem 2.1 from \cite{Keil}.
\end{proof}

Note that it is easy to get this result with $N^{p-3}$ replaced by $N^{p-3}\log N$.
In this appendix we use the same technique to 
prove a $L^p$-estimate for the corresponding one-dimensional exponential sum
\begin{align*}
U_g(\alpha) = \sum_{n \leq N} g(n) e(\alpha n^2),
\end{align*}
where $g: \N \to \C$ is any function with $|g(n)| \leq 1$.

\begin{thm} \label{SA-est}
For $p > 4$, we have
\begin{align*}
\int_{\T} |U_g(\alpha)|^p \,d\alpha \ll_p N^{p-2}.
\end{align*}
\end{thm}

The proof is a simplified version of the proof in \cite{Keil} and we use 
a general result (Theorem \ref{Lp-est} below) on $L^p$-estimates from our previous work.
First we need some notation. Consider the square
\begin{align*}
U_g(\alpha)^2 = \sum_{n_1,n_2 \leq N} g(n_1) g(n_2) e(\alpha (n_1^2 + n_2^2)).
\end{align*}
With $\omega(m) = \#\{n_1,n_2 \leq N: m = n_1^2 + n_2^2\}$ we can write
\begin{align*}
f(m)\omega(m) = \sum_{\substack{n_1,n_2 \leq N \\ n_1^2 + n_2^2 = m}} g(n_1) g(n_2)
\end{align*}
for some function $|f| \leq 1$. This leads to the definition of
\begin{align*}
W_f(\alpha) = U_g(\alpha)^2 = \sum_{m \leq 2N^2} f(m)\omega(m) e(\alpha m).
\end{align*}
For $f \equiv 1$ we get $W(\alpha)$. For an index set $J$ decompose
\begin{align*}
W(\alpha) = \sum_{j \in J} W_j(\alpha) \quad \mbox{ and } \quad
\omega(m) = \sum_{j \in J} \omega_j(m),
\end{align*} 
where $W_j$ is the exponential sum for $\omega_j$.
Define the $L^p$-norms as usual by
\begin{align*}
\|\omega\|_p := \Big(\sum_{m \leq 2N^2} |\omega(m)|^p \Big)^{1/p} \mbox{\ and \ } 
\|W\|_p := \Big(\int_{\T} |W(\alpha)|^p \,d\alpha \Big)^{1/p}_.
\end{align*} 
Now we can state the auxiliary result.

\begin{thm} \label{Lp-est}
For $p > 2$, $N \in \N$ and any $f: \N \to \C$ we have
\begin{align*}
\|W_f\|_p \leq \Big(\sum_{m \leq 2N^2} |f(m)|^2 \omega(m) \Big)^{1/2} 
\Big(\sum_{j \in J} \|W_j\|_{p}^{(p-2)/p} \|\omega_j\|^{2/p}_{2p/(p-2)} \Big)^{1/2}_.
\end{align*} 
\end{thm}

\begin{proof}
This is a special case of Theorem 4.1 from \cite{Keil}.
\end{proof}

The first factor is just a weighted $L^2$-norm of $f$ and is easily estimated in our context.
The second factor needs much more attention and we will perform a variant of 
the major-minor-arc decomposition from the circle method.
For small values of $j$ we have the big major-arc contribitions in $\|W_j\|_{p}$
but the arithmetic counterparts $\omega_j$ are very regular `almost periodic functions'.
For larger values of $j$, the random fluctuations in $\omega_j$ contribute more and 
more to the sum, but are balanced by the savings on the side of the exponential sums $W_j$.
The proposition below gives a quantitative version of these qualitative description.

First we need some notation. Define the local versions of $U(\alpha)$ to be
\begin{align*}
U(q,a) = \sum_{b = 1}^q e(ab^2/q)   \mbox{\quad and \ } v(\alpha) = \int_1^N e(\alpha t^2)\, dt
\end{align*}
and set the major arcs $\Ma$ to be the union of 
\begin{align} \label{eq-Uj}
\Ma(q,a) = \{\beta \in \T: \|\beta - a/q\| \leq Q/N^2 \}.
\end{align}
for $1 \leq a \leq q, (a;q) = 1$ and $q \leq 4Q$. Note that they are disjoint for 
$4Q \leq N^{2/3}$ and we set $Q$ to be a small power of $N$ with $16Q^{6} \leq N$ later.
Define for $Y \leq 2Q$ a dyadic part of the usual major arcs approximation
for quadratic exponential sums (see \cite[Theorem 7.2]{Vau})
\begin{align} \label{eq-Uj}
\Ua_Y(\alpha) := \sum_{Y \leq q < 2Y} q^{-2} \sum_{(a;q) = 1} U(q,a)^2 v(\alpha-a/q)^2 \Big|_{\Ma(q,a)}.
\end{align}
We take the Fourier transform and obtain the arithmetic functions
\begin{align*}
\omega_Y(m) = & \int_{\T} \Ua_Y(\alpha) e(-\alpha m) \,d\alpha \\
= & \sum_{Y \leq q < 2Y} q^{-2} \sum_{(a;q) = 1} U(q,a)^2 e(-am/q) 
\int_{|\beta| \leq Q/N^2} v(\beta)^2 e(-\beta m) \,d\beta,
\end{align*}
which we restrict to $1 \leq m \leq 2N^2$. Set $\omega_Y(m) = 0$ for other values of $m$.
Write $W(\alpha)$ for $U^2(\alpha)$
and define the corresponding exponential sums for $\omega_Y(m)$ as
\begin{align*}
W_Y(\alpha) = \sum_{m \leq 2N^2} \omega_Y(m) e(\alpha m).
\end{align*}
By inserting the definition of $\omega_Y(m)$ we see that $W_Y(\alpha)$ and $\Ua_Y(\alpha)$
are related by the formula
\begin{align} \label{eq-Tint}
W_Y(\alpha) = \int_{\T} \Ua_Y(\beta) L_{2N^2}(\alpha-\beta) \,d\beta,
\end{align}
where $L_M(\alpha) = \sum_{n \leq M} e(\alpha n)$ is the linear exponential sum.

Before we state the main proposition of this section, we set $J$ to be the set
$J = \{1,2,4,\ldots, 2^{D-1},2^D\}$ with $D \in \N$ between $\log_2 Q$ and $1 + \log_2 Q$.
Decompose the function $W$ and $U^2$ into
\begin{align*} 
U^2(\alpha) = \sum_{Y \in J} \Ua_Y(\alpha) +  \Ua'(\alpha) \mbox{\quad and \quad}
W(\alpha) =  \sum_{Y \in J} W_Y(\alpha) +  W'(\alpha).
\end{align*}
One can think of $\Ua'$ as the minor-arc contribution which also contains the approximation error on the major arcs.
Define $\omega'$ as the arithmetic function that belongs to $W'$. 

\begin{proposition}
For $Y \leq 2Q$ we have the estimates
\begin{align*}
\int_{\T} |W_Y(\alpha)|^2 \,d\alpha  & \ll  N^2, \qquad\ \int_{\T} |W'(\alpha)|^2 \,d\alpha \ll N^{2+\epsilon},\\
 \sup_{\alpha} |W_Y(\alpha)| & \ll N^2 Y^{-1}, \qquad \sup_{\alpha} |W'(\alpha)| \ll N^{2+\epsilon}Q^{-1}.
\end{align*}
For each $k \in \N$ with $Q^{8k} \leq N$ we have
\begin{align*} 
\sum_{m \leq 2N^2} |\omega_Y(m)|^{2k} \ll_{\epsilon,k} Y^{\epsilon} N^2 \mbox{\quad and \ } 
\sum_{m \leq 2N^2} |\omega'(m)|^{2k} \ll_{\epsilon,k} N^{2+\epsilon}.
\end{align*}
\end{proposition}

\begin{proof}
We go through the estimates one by one.
For the first one, we observe that by \eqref{eq-Tint} the function $W_Y$ is a projection of $\Ua_Y$ and, 
therefore, by Bessel's inequality, we have the upper bound
\begin{align*}
\int_{\T} |W_Y(\alpha)|^2 \,d\alpha \leq \int_{\T} |\Ua_Y(\beta)|^2 \,d\beta,
\end{align*}
which can also be established directly. 
Insert the definition of $\Ua_Y$ from \eqref{eq-Uj} and expand. We obtain due to the disjointness 
of the sets $\Ma(q,a)$ the evaluation
\begin{align*}
\int_{\T} |\Ua_Y(\beta)|^2 \,d\beta 
= \sum_{Y \leq q < 2Y} q^{-4} \sum_{(a;q) = 1} |U(q,a)|^4 \int_{\Ma(q,a)} |v(\alpha-a/q)|^4 \,d\alpha.
\end{align*}
By the well known estimates $|U(q,a)| \ll q^{1/2}$ (see for example Lemma A.5 in \cite{Keil}) 
and $|v(\beta)| \ll N(1 + |\beta|N^2)^{-1/2}$
(see \cite[Theorem 7.3]{Vau}), we obtain the claim by a straightforward calculation.\\ 
The $L^2$-bound for $W'$ follows from the previous bound since by Parseval's identity
and $\omega(m) \ll_{\epsilon} m^{\epsilon}$ we get
\begin{align*}
\int_{\T} |W(\alpha)|^2 \,d\alpha = \sum_{m \leq 2N^2} \omega(m)^2 \ll N^{2+\epsilon}.
\end{align*}

The $L^{\infty}$-estimate for $W_Y$ starts with \eqref{eq-Uj} and \eqref{eq-Tint} and gives us
\begin{align*}
|W_Y(\alpha)|  & \leq  
\sum_{Y \leq q < 2Y} q^{-2} \sum_{(a;q) = 1} |U(q,a)|^2 \int_{\Ma(q,a)} |v(\beta-a/q)|^2 |L_{2N^2}(\alpha-\beta)| \,d\beta\\
& \leq Y^{-1} \sum_{Y \leq q < 2Y} \sum_{(a;q) = 1} \int_{\Ma(q,a)} |v(\beta-a/q)|^2 |L_{2N^2}(\alpha-\beta)| \,d\beta
\end{align*}
using again the estimate $|U(q,a)| \ll q^{1/2}$.
For a given pair of $q$ and $a$ we can estimate the inner integral by Cauchy's inequality.
The two resulting integrals can be dealt with the estimate $|v(\beta)| \ll N(1 + |\beta|N^2)^{-1/2}$
and Parseval's identity to obtain the bound
\begin{align*}
& \int_{\Ma(q,a)} |v(\beta-a/q)|^2 |L_{2N^2}(\alpha-\beta)| \,d\beta\\ 
\leq & \Big (\int_{\Ma(q,a)} |v(\beta-a/q)|^4 \,d\beta \Big)^{1/2} 
\Big(\int_{\T} |L_{2N^2}(\alpha-\beta)|^2 \,d\beta \Big)^{1/2} \ll N^2.
\end{align*}
This estimate is very wasteful if $a/q$ is `far' away from $\alpha$ and we use it only when 
$\|a/q-\alpha\| \leq 2Q^4/N^2$. But there is only one such pair $a$ and $q$ 
since $\|a_i/q_i-\alpha\|\leq 2Q^4/N^2$ for $i \in \{1,2\}$ implies
that $1/q_1q_2 \leq \|a_1/q_1-a_2/q_2\|\leq 4Q^4/N^2$, which isn't possible due to the restriction
$16Q^{6} \leq N$.

For all the other values of $a$ and $q$ we can do better by using 
the bounds $|L(\alpha)| \leq \|\alpha\|^{-1}$ and $|v(\beta)| \leq N$. Since $\|a/q-\alpha\| > 2Q^4/N^2$
and $\|a/q-\beta\| \leq Q/N^2$ we get $|L_{2N^2}(\alpha-\beta)| \leq N^2/Q^4$. This implies
\begin{align*}
& \int_{\Ma(q,a)} |v(\beta-a/q)|^2 |L_{2N^2}(\alpha-\beta)| \,d\beta \ll \mu(\Ma(q,a)) N^3/Q^4,
\end{align*}
where $\mu$ is the Lebesgue measure.
Summing over $a$ and $q$ and using the bound $\mu(\Ma(q,a)) \leq 2Q/N^2$, we obtain 
the $L^\infty$-estimate for $W_Y$.

For the last part of the $L^\infty$-estimates write
\begin{align*}
\Ua^*(\alpha) = \sum_{Y \in J} \Ua_Y(\alpha)
\end{align*}
as an abbreviation. From $|L_{2N^2}(\alpha)| \leq \min\{2N^2,\|\alpha\|^{-1}\}$
we obtain 
\begin{align*}
\int_{\T} |L_{2N^2}(\alpha)|\,d\alpha \ll \log N.
\end{align*}
We use this observation together with \eqref{eq-Tint} and the `projection identity'
\begin{align*}
W(\alpha) = \int_{\T} U^2(\beta) L_{2N^2}(\alpha-\beta) \,d\beta
\end{align*}
to reduce the $L^{\infty}$-bound for $W'$ to
\begin{align*}
|W'(\alpha)| \leq \int_{\T} |U^2(\beta)-\Ua^*(\beta)| |L_{2N^2}(\alpha-\beta)| \,d\beta
\leq \log N \sup_{\beta \in \T} |U^2(\beta)-\Ua^*(\beta)|.
\end{align*}

Since $\Ua^*$ is the major arc approximation for $U^2$ and zero outside of $\Ma$, we can use 
Theorem 7.2 from \cite{Vau} to estimate the approximation error $\Delta = U(\alpha) - q^{-1}U(q,a)v(\alpha-a/q)$
on the major arcs by $|\Delta| \ll q(1 + |\beta|N^2) \ll Q^2$. 
This is acceptable for our choice of $Q$. We estimate the size of $U$ by 
Weyl's inequality (see \cite[Lemma 2.4]{Vau}) from above by
\begin{align*}
|U(\alpha)| \ll N^{1 + \epsilon} (1/q + 1/N + q/N^2)^{1/2},
\end{align*}
if $\alpha = a/q + \beta$ with $|\beta| \leq 1/q^2$.
For $\alpha \notin \Ma$ we have $q \geq Q$ and the estimate
$|U(\alpha)| \ll N^{1 + \epsilon} Q^{-1/2}$ as long as $Q \leq N$,
which gives the desired estimate on the minor arcs.

The result for $\omega_Y$ is obtained from the decomposition
\begin{align*}
\omega_Y(m) =  \sum_{Y \leq q < 2Y} q^{-2} \sum_{(a;q) = 1} 
U(q,a)^2 e(-am/q) \int_{|\beta| \leq Q/N^2} v(\beta)^2 e(-\beta m) \,d\beta.
\end{align*}
Since this factors into an analytic and an arithmetic part, we can use the Cauchy-Schwarz-inequality
to estimate the $2k$-th moment over each part separately but with $4k$ instead of $2k$.
The analytic part can be dealt with by the Hausdorff-Young inequality for $p = (1-1/(4k))^{-1}$
and contributes
\begin{align*}
\sum_{m \leq 2N^2} \Big| \int_{|\beta| \leq Q/N^2} v(\beta)^2 e(-\beta m) \,d\beta \Big|^{4k}
\ll_k \Big( \int_{\T} |v(\beta)|^{2p} \,d\beta \Big)^{4k/p} \ll N^{2}.
\end{align*}
On the other hand, the arithmetic moment can be bounded by \cite[Lemma 5.3]{Keil}
and we obtain
\begin{align*}
\sum_{m \leq 2N^2} \Big|\sum_{Y \leq q < 2Y} q^{-2} \sum_{(a;q) = 1} 
U(q,a)^2 e(-am/q) \Big|^{4k} \ll_{k,\epsilon} N^2 Y^{\epsilon}
\end{align*}
as lond as $Y^{8k} \leq 2N^2$,
which is satisfied, if $Q$ is a sufficiently small power of $N$.

The result for $\omega'= \omega - \sum_{Y \in J} \omega_Y$ follows from by H\"older's inequality 
and the fact that the number of solutions to $m = x^2 + y^2$ is bounded by $m^{\epsilon}$.
\end{proof}

\begin{proof}[Proof of Theorem \ref{SA-est}.]
We use the proposition above with Theorem \ref{Lp-est}.
We set $W_f = U_g^2$, $N_1 = 2N^2$, $J_0 = \{0\} \cup J$, where $W_0 = W'$, $\omega_0 = \omega'$
and $r = p/2 > 2$.
The first factor in
\begin{align*}
\|W_f\|_r \leq \Big(\sum_{m \leq 2N^2} |f(m)|^2 \omega(m) \Big)^{1/2} 
\Big(\sum_{Y \in J_0} \|W_Y\|_{r}^{(r-2)/r} \|\omega_Y\|^{2/r}_{2r/(r-2)} \Big)^{1/2}
\end{align*} 
is $O(N)$ since $|f(m)| \leq 1$ and $\omega(m) = \#\{n_1,n_2 \leq N: m = n_1^2 + n_2^2\}$.
The bound for the $L^r$-norm of $W_Y$ for $Y \neq 0$ follows from
\begin{align*}
\int_0^1 |W_Y(\alpha)|^r \,d\alpha \leq \sup_{\alpha} |W_Y(\alpha)|^{r-2} \int_0^1 |W_Y(\alpha)|^2 \,d\alpha
\ll (N^2 Y^{-1})^{r-2} N^2
\end{align*} 
by the first part of the proposition. For $Y = 0$ we obtain in the same way
\begin{align*}
\int_0^1 |W_0(\alpha)|^r \,d\alpha \ll (N^{2+\epsilon} Q^{-1})^{r-2} N^{2+\epsilon}.
\end{align*} 
The moment estimates for $\omega_Y$ give
\begin{align*}
\sum_{m \leq 2N^2} |\omega_Y|^{2r/(r-2)} \ll N^2 Y^{\epsilon} \quad \mbox{ and } \quad 
\sum_{m \leq 2N^2} |\omega_0|^{2r/(r-2)} \ll N^{2+\epsilon}.
\end{align*} 
If we parametrize $J$ by $Y = 2^i$, we get for $D \approx \log_2 N$ the upper bound
\begin{align*}
 \sum_{Y \in J_0} \|W_Y\|_{r}^{(r-2)/r} \|\omega_Y\|^{2/r}_{2r/(r-2)}
\leq & \sum_{i \leq D} \Big( (N^2 2^{-i})^{r-2} N^2 \cdot 2^{\epsilon i} N^2 \Big)^{(r-2)/r^2} + \\ 
+&  \Big(  (N^{2+\epsilon}Q^{-1})^{r-2} N^{2+\epsilon} \cdot N^{2+\epsilon} \Big)^{(r-2)/r^2}.
\end{align*}
This is $O(N^{2(r-2)/r})$ if $r > 2$ and $Q$ is a small power of $N$ dependent on $r$.
Combine the square-root of this with the first factor to get
\begin{align*}
\|W_f\|_r \ll N^{(2r-2)/r},
\end{align*}
which gives the result when we take $r$-th powers and substitute $r = p/2$.
\end{proof}

We give a corollary of Theorem \ref{SA-est} for the two-dimensional version.

\begin{lem} \label{Vg4-est}
Let $|g| \leq 1$, then for $p > 4$ we have
\begin{align*}
\sup_{\beta} \int_{\T} |V_g(\alpha,\beta)|^p \,d\alpha \ll N^{p-2}.
\end{align*}
\end{lem}

\begin{proof}
Write
\begin{align*}
V_{g}(\alpha, \beta) = \sum_{n \leq N} g(n) e(\alpha n^2 + \beta n) = 
\sum_{n \leq N} g(n)e(\beta n) e(\alpha n^2) = U_h(\alpha) 
\end{align*}
for $h(n) = g(n) e(\beta n)$. The estimate now follows from Theorem \ref{SA-est}.
\end{proof}


\begin{thebibliography}{HD}



\bibitem{Aig}
M. Aigner, 
\newblock {\em Combinatorial theory.}
\newblock{ Springer-Verlag, Berlin-New York, 1979.}

\bibitem{BDLW}
J. Br\"udern, R. Dietmann, J. Y. Liu and T. D. Wooley,
\newblock {\em A Birch-Goldbach theorem.}
\newblock{ Arch. Math. (Basel) 94 (2010), no. 1, 53--58}

\bibitem{Bour2}
J. Bourgain,
\newblock {\em Fourier transform restriction phenomena for certain lattice subsets and applications 
to nonlinear evolution equations. I. Schr\"odinger equations.}
\newblock{ Geom. Funct. Anal.  3  (1993),  no. 2, 107--156.}

\bibitem{Bour3}
J. Bourgain,
\newblock {\em Roth's theorem on progressions revisited.}
\newblock{ J. Anal. Math.  104  (2008), 155--192. }

\bibitem{Dav}
H. Davenport,
\newblock {\em Analytic methods for Diophantine equations and Diophantine inequalities. Second edition.}
\newblock{ Cambridge University Press, Cambridge, 2005.}

\bibitem{ErTu}
P. Erd\H{o}s, P. Turan
\newblock {\em On some sequences of integers.}
\newblock{ J. Lond. Math. Soc. 11, 261--264 (1936).}

\bibitem{Fur}
H. Furstenberg, Y. Katznelson and D. Ornstein,
\newblock {\em The ergodic theoretical proof of Szemer\'edi's theorem.}
\newblock{ Bull. Amer. Math. Soc. (N.S.)  7  (1982), no. 3, 527--552. }

\bibitem{Gow}
T. Gowers,
\newblock {\em A new proof of Szemer\'edi's theorem. }
\newblock{ Geom. Funct. Anal.  11  (2001),  no. 3, 465--588.}

\bibitem{Gr}
B. Green,
\newblock {\em Roth's theorem in the primes.}
\newblock{ Ann. of Math. (2)  161  (2005),  no. 3, 1609--1636.}

\bibitem{GrTao1}
B. Green and T. Tao,
\newblock {\em Restriction theory of the Selberg sieve, with applications.}
\newblock{ J. Th\'eor. Nombres Bordeaux  18  (2006),  no. 1, 147--182.}

\bibitem{GrTao2}
B. Green and T. Tao,
\newblock {\em The primes contain arbitrarily long arithmetic progressions.}
\newblock{ Ann. of Math. (2) 167 (2008), no. 2, 481--547. }

\bibitem{GrTao3}
B. Green and T. Tao,
\newblock {\em Linear equations in primes}
\newblock{ Ann. of Math. (2)  171  (2010),  no. 3, 1753--1850.}

\bibitem{HB}
D. R. Heath-Brown, 
\newblock {\em Integer sets containing no arithmetic progressions.}
\newblock{ J. London Math. Soc. (2)  35  (1987),  no. 3, 385--394. }

\bibitem{HB2}
D. R. Heath-Brown, 
\newblock {\em A new form of the circle method, and its application to quadratic forms.}
\newblock{ J. Reine Angew. Math. 481 (1996), 149--206.}

\bibitem{Keil}
E. Keil, 
\newblock {\em On a diagonal quadric in dense variables.}
\newblock{Preprint.}

\bibitem{Liu}
J. Liu, 
\newblock {\em Integral points on quadrics with prime coordinates.}
\newblock{ Monatsh. Math.  164  (2011),  no. 4, 439--465.}

\bibitem{LPW}
L. Low, J. Pitman, A. Wolff. 
\newblock {\em Simultaneous diagonal congruences.}
\newblock{  J. Number Theory  29  (1988),  no. 1, 31--59.}

\bibitem{Roth}
K. F. Roth, 
\newblock {\em On certain sets of integers.}
\newblock{ J. London Math. Soc.  28,  (1953), 104--109.}

\bibitem{Roth2}
K. F. Roth, 
\newblock {\em On certain sets of integers. II.}
\newblock{ J. London Math. Soc. 29, (1954). 20--26.}

\bibitem{San}
T. Sanders,
\newblock {\em On Roth's theorem on progressions.}
\newblock{ Ann. of Math. (2)  174  (2011),  no. 1, 619--636. }

\bibitem{Smith}
M. Smith, 
\newblock {\em On solution-free sets for simultaneous quadratic and linear equations.}
\newblock{ J. Lond. Math. Soc. (2)  79  (2009),  no. 2, 273--293.}

\bibitem{Szem}
E. Szemer\'edi,
\newblock {\em On sets of integers containing no $k$ elements in arithmetic progression.}
\newblock{ Acta Arith.  27  (1975), 199--245.}

\bibitem{Vau}
R. C. Vaughan, 
\newblock {\em The Hardy-Littlewood method. Second edition.}
\newblock{ Cambridge University Press, Cambridge, 1997.}

\end{thebibliography}
\end{document}